\documentclass[a4paper,12pt]{article}
\usepackage[utf8]{inputenc}

\usepackage{amsmath,amsthm,amscd,amssymb,eucal,mathrsfs}
\usepackage{amsfonts}
\usepackage{latexsym}
\usepackage{graphicx} 
\usepackage{color}
\usepackage{multicol}
\usepackage{dsfont}
\usepackage[all]{xy}

\newtheorem{thm}{Theorem}[section]

\newtheorem{lem}[thm]{Lemma}
\newtheorem{cor}[thm]{Corollary}
\newtheorem{rem}[thm]{Remark}
\newtheorem{dfn}[thm]{Definition}

\newtheorem{exa}[thm]{Example}

\newtheorem{quest}[thm]{Question}
\newtheorem*{Satz*}{Satz}
\newtheorem{task}{Task}

\newtheorem{ass}{Assumption}

\newcommand{\mathset}[1]{{\left\{#1\right\}}}
\newcommand{\absolute}[1]{\left\lvert#1\right\rvert}
\newcommand{\norm}[1]{\left\|#1\right\|}

\DeclareMathOperator{\Star}{Star}
\DeclareMathOperator{\id}{id}

\DeclareMathOperator{\PGL}{PGL}
\DeclareMathOperator{\SL}{SL}

\DeclareMathOperator{\Cov}{Cov}
\DeclareMathOperator{\dom}{dom}
\DeclareMathOperator{\Res}{Res}
\DeclareMathOperator{\Tr}{Tr}
\DeclareMathOperator{\Sp}{Sp}
\DeclareMathOperator{\Max}{Max}
\DeclareMathOperator{\Red}{Red}

\title{Heat equations and wavelets on Mumford curves}

\author{Patrick Erik Bradley}

\date{\today}

\begin{document}

\maketitle


\begin{abstract}
A general class of heat operators over non-archimedean local fields
acting on $L_2$-function spaces on affinoid domains in the local field are developed. $L_2$-spaces and integral operators invariant under the action of a non-archimedean Schottky group are constructed in order to have
nice function spaces and operators on
Mumford curves. General wavelets are constructed which, together with functions coming from certain graph Laplacian eigenvectors, diagonalise these operators.
The corresponding Cauchy problems for the heat equations with these operators are solved in the affirmative, and properties of wavelet eigenvalues and the spectral gaps of the operators on Mumford curves are studied.
\end{abstract}

\section{Introduction}

Heat equations in different flavours are of great interest from various different viewpoints. In an Archimedean setting, they are extensively  treated in the literature. 
Ever since the forthcoming of the Taibleson-Vladimirov operator, a pseudodifferential operator on a non-archimedean local field \cite{Taibleson1975,VVZ1994}, there has been active research on heat equations on such a field, in particular the field of $p$-adic numbers, which studies this operator or generalisations of it.
Whereas in the classical case, heat equations are also extensively studied on manifolds, there is according to  \cite{CZ2018} no comparable theory of pseudodifferential operators over $p$-adic manifolds, not to say over manifolds defined over a non-archimedean local field. A construction of a certain pseudodifferential operator on a certain $p$-adic manifold was undertaken in \cite{GenDiffMg}, however the generalised diffusion obtained cannot be considered a heat equation. These operators are invariant under the action of certain finite groups. A study of how $p$-adic pseudodifferential operators transform under group actions on the Bruhat-Tits tree was undertaken in \cite{AK2010}.
A different approach was pursued in \cite{Zuniga2020}, where $p$-adic integral operators on closed open (clopen) subsets of $\mathds{Q}_p$ were constructed which are direct analogues of graph Laplacians. In fact, they can be seen as $p$-adic forms of graph Laplacians by the dictionary developed in \cite{pWaveGraph}. 

\smallskip
A `nice' theory of $p$-adic heat equations is one in which the heat operator is diagonalisable by the $p$-adic Fourier transform \cite{Zuniga2015}. More general  operators can be diagonalised using wavelets on ultrametric spaces \cite{XK2005}.
Kozyrev's well-known $p$-adic wavelets were found to diagonalise the Taibleson-Vladimirov operator \cite{Kozyrev2002}.
However, the graph-based heat operators are not entirely diagonalisable by Kozyrev wavelets, one also needs the eigenvectors of the graph Laplacian \cite{Zuniga2020}. The work of W.\ Z\'{u}\~{n}iga-Galindo contains many classes of $p$-adic heat equations, for which the Cauchy problem is solved in the affirmative. This includes those with graph-based operators, like also those in \cite{TZ2018}.

\smallskip
The non-archimedean counterpart of Riemann surfaces are the \emph{Mumford curves} which are projective algebraic curves 
defined over non-archimedean fields allowing a Schottky uniformisation \cite{Mumford1972}. Locally, they are holed discs inside the base field. That means that if the base field is a non-archimedean local field, then the Haar measure allows integration of functions defined on these local pieces. The question is, how to glue together local operators on overlaps in a meaningful way.
The stable reduction theorem \cite{DM1969} states that there is a model over a finite extension of the base field such that the special fibre, aka the reduction curve over the residue field, is a singular projective curve whose irreducible components are all rational curves, and the singularities are ordinary double points.
Consequently, the intersection graph of the reduction curve of a Mumford curve has first Betti number equal to the genus of the curve.
The rigid analytic theory of Mumford curves \cite{GvP1980} allows to construct the intersection graph with the help of a covering of the curve by holed discs. This fact, together with the rigid analytic proof of the stable reduction theorem from \cite{FP2004} gives insight into how to obtain an integral operator on Mumford curves. Namely, these are
 graph-based integral operators which generalise both the Taibleson-Vladimirov operator as well as Z\'{u}\~{n}iga-Galindo's operators on the set of $K$-rational points of a Mumford curve by taking disjoint covers refining a covering with holed discs, and where each such disjoint piece corresponds to a vertex in a reduction graph of the curve. This is the first contribution of the present article. 
 
 \smallskip
The integral operator constructed here has a kernel which is doubly invariant under the Schottky group by summing up twists of an operator defined on the fundamental domain of the action of the Schottky group.
 In order to obtain a meaningful space of $L_2$-functions on a 
 Mumford curve, a weighted integral is defined in order to take care of the derivative of each M\"obius transformation involved through the integral transformation rule. 
A decomposition of the space of of invariant $L_2$ functions on the universal covering space of a Mumford curve is obtained through first defining a very general class of wavelets which extends Kozyrev's wavelets to such which are supported on sets which need not be discs.
The $L_2$-functions on the fundamental domain (a holed disc) then decompose into eigenfunctions of our graph-based operators which consist of such wavelets and of functions coming from eigenvectors of a graph Laplacian.

\smallskip
The next step is then to construct from these wavelets such which are invariant under the Schottky group and then prove an analoguous structure Theorem for the invariant $L_2$-spaces.
After proving that the operator is self-adjoint and bounded, the next result is 
the statement that the Cauchy problem for the heat equation for this operator and the invariant operator has a solution which is unique.
The proof follows the methods from \cite{Zuniga2020}.
Finally, some statements about the wavelet eigenvalues of the heat operators on Mumford curves are proven. In particular, the spectral gap of the new operators
can be arbitrarily small in families of Mumford curves with isomorphic stable reduction graphs. An example is given, where the spectral gap is a wavelet eigenvalue, but
it is an open question, whether the spectral gap can also be a graph Laplacian eigenvalue.

\bigskip
This article is subdivided into six numbered sections, the present one being the introduction.
The following Section 2 fixes some notation. This is followed
in Section 3 by a brief introduction to the rigid analytic theory of Mumford curves
and affinoid domains as far as is needed for the remaining parts.
Section 4 generalises the dictionary developed in \cite{pWaveGraph}
to the more general situation of this article and now contains a correspondence between matrix eigenvalues and operator eigenvalues.
Section 5 develops the wavelets and integral  operators
on clopen pieces of the local field and proves the structure theorem for the $L_2$-spaces and the Cauchy problem for the heat equation.
Section 6 makes everything from the previous section invariant under the action of the Schottky group and thus obtains a positive answer to the Cauchy problem for the heat equation on Mumford curves.
It concludes with a study of wavelet eigenvalues for Mumford curves in general and Tate elliptic curves in particular.

\section{Notation}
Let $K$ be a non-archimedean local field whose absolute value is denoted as $\absolute{\cdot}_K$. 
The multiplicative group of $K$ is denoted as $K^\times$.
The unit ball of $K$ is denoted as $O_K$. The Haar measure $\mu$ on $K$ is chosen such that $\mu(O_K)=1$. 
It is known that $O_K$ is a discrete valuation ring whose uniformiser is denoted as $\pi$. It has the property
\[
\absolute{\pi}_K=p^{-\frac{1}{e}}
\]
for some $e\ge 1$, and where $p$ is a prime number. 
In the case that $K$ is a $p$-adic number field of degree $n$ over $\mathds{Q}_p$, then there is the well-known formula
\[
n=e\cdot f
\]
where $f$ is the degree of the residue field $k$ over 
the finite field $\mathds{F}_p$ with $p$ elements.
Notice that, with our choice of Haar measure,
\[
\mu(\pi^mO_K)=p^{-mf}
\]
for $m\in\mathds{Z}$.
Let 
\[
\tau\colon k\to O_K
\]
be a lift of the residue map which takes a residue class modulo $\pi O_K$ to a representative in $O_K$.
The Bruhat-Tits tree for the projective-linear group $\PGL_2(K)$ will be denoted as $\mathcal{T}_K$. Its vertices are in one-one correspondence with discs of the projective line $\mathds{P}^1(K)$, where a disc is a subset of the form
\[
\mathset{\absolute{x-a}_K\le r}\quad\text{or}\quad
\mathset{\absolute{x-a}_K\ge r}\cup\mathset{\infty}
\]
where $a\in K$ and $r\in\absolute{K^\times}$. We will fix an additive character 
\[
\chi\colon K\to\mathds{C}^\times
\]
When integrating a complex-valued function $f(x)$ on $K$, the Haar measure will be denoted as $dx$, so that the integral has the form
\[
\int_Kf(x)\,dx
\]
if it exists. Finally, we will extensively make use of indicator functions which we will write as
\[
\Omega(x\in U)=\begin{cases}
1,&x\in U\\
0,&\text{otherwises}
\end{cases}
\]
where $U$ is a measurable subset of $K$.

\section{Mumford curves and their reductions}

This section is an introduction into the theory of rigid analytic geometry and Mumford curves. Readers familiar with these topics can simply browse through this section in order to pick up necessary notation and viewpoints for the later parts of this article.

\subsection{Mumford curves}

Mumford curves were first constructed in \cite{Mumford1972} as a successful attempt to generalise Tate's analytic uniformisation of $p$-adic elliptic curves, which can be found in \cite{Roquette1970}.
Through this construction, Mumford revealed a one-to-one correspondence between conjugacy classes of Schottky groups in $\PGL_2(K)$ and a certain class of projective algebraic curves defined over $K$.

\smallskip
In order to present more details of this correspondence,  we will give a brief summary of \cite[Ch.\ I, III, IV, V]{GvP1980}, adapted to the setting of a non-archimedean local field.
The slightly over two pages in \cite[Ch.\ 5.4]{FP2004}
 also contain a brief overview over Mumford curves.

\smallskip
A M\"obius transformation $\gamma\in\PGL_2(K)$ is called
\emph{hyperbolic}, if
\[
\absolute{\Tr(A_\gamma)}_K^2>1
\]
where $A_\gamma\in\SL_2(K)$ is the representative of $\gamma$ as a special linear matrix.

\smallskip
A discrete subgroup of $\PGL_2(K)$ which is freely generated by $g\ge 1$ hyperbolic transformations is called a \emph{Schottky group}.

\smallskip
A \emph{limit point} of a subgroup $\Gamma$ of $\PGL_2(K)$ is a point $x\in\mathds{P}^1(K)$ such that there exists a point $x_0\in\mathds{P}^1$ and a sequence $\gamma_n\in \Gamma$ such that
\[
\lim\limits_{n\to\infty}\gamma_n(x_0)=x
\]
The set of limit points is denoted as $\mathscr{L}(\Gamma)$,
and $\Omega=\mathds{P}^1(K)\setminus\mathscr{L}(\Gamma)$ is the set of \emph{regular points} of $\Gamma$.
A Schottky group $\Gamma$ has the property 
\[
\Omega\neq\emptyset
\]
This is an open subset of $\mathds{P}^1(K)$.

\smallskip
The \emph{Mumford curve} associated with a Schottky group $\Gamma$
with $g\ge1$ generators
is the quotient space
\[
X=\Omega/\Gamma
\]
It is a projective algebraic curve over $K$ of genus $g$ \cite[Ch.\ III]{GvP1980}.

\smallskip
A \emph{rational affinoid domain} is a holed disc in $\mathds{P}^1$.
A Mumford curve has a finite covering $\mathcal{U}$
by rational affinoid domains \cite[Ch.\ V]{GvP1980}.
In this covering, two overlapping patches $U\in\mathcal{U}$ are glued with another along a boundary component (which is a circle).

\subsection{Affinoid domains in $K$ and their reductions}\label{sec:Red}

The \emph{standard Tate algebra} $T_n$ is defined as
\[
T_n=K\langle z_1,\dots,z_n\rangle
\]
is the subring of the ring $K[[z_1,\dots,z_n]]$ of all formal power series consisting of those power series
\[
f=\sum\limits_\alpha c_\alpha z_1^{\alpha_1}\cdots z_n^{\alpha_n}
\in K[[z_1,\dots,z_n]]
\]
such that $\lim c_\alpha=0$, where $\alpha=(\alpha_1,\dots,\alpha_n)\in\mathds{N}^n$.
The \emph{Gauss norm} on $T_n$ is defined as
\[
\norm{f}=\max\absolute{c_\alpha}_K
\]
One further has
\[
T_n^o=\mathset{f\in T_n\colon \norm{f}\le 1},
\quad T_n^{oo}=\mathset{f\in T_n\colon \norm{f}<1}
\]
and one has the canonical isomorphism
\[
\bar{T}_n=T_n^o/T_n^{oo}\cong k[z_1,\dots,z_n]
\]
There are spaces attached to these objects, namely
\[
X=\Sp(T_n)=\mathset{(x_1,\dots,x_n)\in (K^{alg})^n\colon \max\limits_{i=1}^n\absolute{x_i}_{K^{alg}}\le 1}
\]
and 
\[
\Max(\bar{T}_n)=\mathds{A}^n(k^{alg})
\]
there is a natural surjective map
\[
\Red_X\colon \Sp(T_n)\to\Max(\bar{T}_n)
\]
called \emph{reduction map}.

\smallskip
An \emph{affinoid algebra} $A$ over $K$ is a $K$-algebra for which there exists a $K$-algebra homomorphism
\[
T_n\to A
\]
such that $A$ is a finitely generated $T_n$-module.
By \cite[Thm.\ 3.2.1]{FP2004}, every affinoid algebra $A$ is of the form
\[
A=T_m/I
\]
for some ideal $I$ in the Tate algebra $T_m$, and the Gauss norm of $T_m$ induces a norm on $A$ for which it is a Banach algebra over $K$.
One defines the \emph{affinoid space}
\[
X=\Sp(A)=\mathset{\text{maximal ideals in $A$}}
\]
In general, if $m\in\Sp(A)$, then the field
$A/m$ is a finite extension of $K$.

\smallskip
We are interested in spaces of the form
\[
X(K)=\mathset{m\in X\colon A/m\cong K}
\]
of \emph{$K$-rational points} of $X=\Sp(A)$.

\begin{exa}
A disc 
\[
B_r(a)=\mathset{x\in K\colon \absolute{x-a}_K\le \absolute{r}_K} 
\]
with $r\in K^\times$ can be written as
\[
X(K)
\]
with $X=\Sp(A)$, and
where $A$ is the affinoid algebra
\[
A=K\left\langle \frac{x-a}{r}\right\rangle
\]
\end{exa}

\begin{exa}
The annulus
\[
Y=\mathset{x\in K^{alg}\colon \absolute{\pi}_K\le\absolute{x}_K\le1}
=\Sp B
\]
with
\[
B=K\left\langle \frac{\pi}{x},x\right\rangle
=K\langle S,T\rangle/(ST-\pi)
\]
Observe that
\[
\overline{B}=k[S,T]/(ST)=k[X,X^{-1}]
\]
Hence, the reduction of $Y$ is the union of two affine lines intersecting in one ordinary double point $P_d$.
Let us verify that
\[
\Red^{-1}(P_d)=\mathset{x\in K^{alg}\colon\absolute{\pi}_K<\absolute{x}_K<1}=:Y^-
\]
Namely, the fibre of $P_d$ can be written as
\[
S\equiv T\equiv 0\mod \pi \quad\Leftrightarrow\quad
\absolute{S}_K<1,\;\absolute{T}_K<1
\]
And because $T=\frac{\pi}{S}$, we have
\[
\absolute{T}_K=\frac{\absolute{\pi}_K}{\absolute{S}_K}<1
\quad\Leftrightarrow\quad\absolute{S}_K>\absolute{\pi}_K
\]
Together, this yields $\absolute{\pi}_K<\absolute{S}_K<1$, as desired.

\smallskip
Notice that 
\[
Y^-(K)=\emptyset
\]
i.e.\ the $K$-rational annulus $Y(K)$ is covered by the union of two disjoint circles.
This will be important later in this article.
\end{exa}

Let $X=B_{r_0}(a_0)\setminus\bigcup\limits_{i=1}^n B_{r_i}(a_i)$ be a holed disc with $r_i\in\absolute{K^\times}_K$,
$a_i\in B_r(a_0)$, and the holes $B_{r_i}(a_i)$ pairwise disjoint, $i=1,\dots,n$. 
The reduction tree of $X$ can be obtained in the following way as a subtree of the Bruhat-Tits tree $\mathcal{T}_K$ as follows:
Let $b_0,b_1,\dots,b_n\in K$ such that
\[
\absolute{a_i-b_i}=r_i,\quad i=0,\dots,n
\]
and let $S=\mathset{a_0,b_0,\dots,a_n,b_n,\infty}$ and
$T^*(S)$ the smallest subtree of $\mathcal{T}_K$ whose ends are the set $S$. Then the inner part $T(X)$, i.e.\ the subtree whose vertices are
\[
\mathset{v(a,b,c)\colon \mathset{a,b,c}\subset S,\;\absolute{\mathset{a,b,c}}=3}
\]
is the reduction tree of $X$. Here, $v(a,b,c)$ denotes the unique vertex of $\mathcal{T}_K$ determined by the three ends $a,b,c\in\mathds{P}^1(K)$.

\smallskip
Notice further that an edge in $T(S)$ corresponds to a double point in the reduction curve $\overline{X}$, respectively to its fibre under the reduction map, i.e.\ an open annulus in $K^{alg}$.

\smallskip
If one instead considers the edges of $T(S)$ to be geodesic paths in the Bruhat-Tits tree $\mathcal{T}_K$ and defines $T_X$ as the smallest  subtree having only edges from $\mathcal{T}_K$ defined by these geodetic paths, then one has the 
 following remark which is well-known in the theory of rigid analysis:
\begin{rem}\label{treepatch}
A vertex $v$ in $T(X)$ corresponds to a holed disc $U_v$ in $K$ obtained by removing from a disc one maximal subdisc per edge in $T(X)$ attached to $v$.
An edge $v$ in $T(X)$ corresponds to an annulus in $X$ without any $K$-rational points.
\end{rem}

\begin{rem}\label{verticialCovering}
A consequence of Remark \ref{treepatch} is that the collection 
of the $U_v$ where $v$ runs through the vertices of $T(X)$ yields 
a disjoint covering of the set $X(K)$ of $K$-rational points of $X$.
We call such a covering of $X(K)$ \emph{verticial}. It is uniquely determined by the tree $T(X)$.
\end{rem}

\begin{exa}
Let $X=\mathset{\absolute{\pi}_K\le\absolute{x}_K\le 1}$. Then
we can take $S=\mathset{0,\pi,1,\infty}$, i.e.\
\[
a_0=0,\;b_0=\pi,\;a_1=0,\;b_1=1
\]
and obtain the following tree $T^*(S)$:
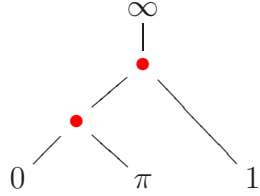
\[
\xymatrix@=10pt{
&&\infty&\\
&&\textcolor{red}{\bullet}\ar@{-}[u]\ar@{-}@[red][dl]\ar@{-}[ddrr]&&\\
&\textcolor{red}{\bullet}\ar@{-}[dl]\ar@{-}[dr]&&\\
0&&\pi&&1
}
\]
and the reduction tree $T(S)$  of $X$ is the red segment and equals $T_X$.
The verticial covering of $X(K)$ consists of two disjoint holed discs corresponding to the two red vertices. They are given as follows:
\[
U=\mathset{\absolute{x}_K=1},\quad V=\mathset{\absolute{x}_K=\absolute{\pi}_K}
\]
Observe that the open annulus
\[
\mathset{\absolute{\pi}_K<\absolute{x}_K<1}
\]
corresponding to the red edge does not have any $K$-rational points.
\end{exa}

\subsection{Mumford curves as rigid analytic spaces}

Mumford curves are obtained by glueing finitely many holed discs in such a way that one obtains a rational affinoid covering $\mathcal{U}$
which is \emph{pure} in the sense that for all $U,V\in\mathcal{U}$ it holds true that

\begin{enumerate}
\item If $U\cap V\neq\emptyset$, then there exists an affine subset $Z\subset\overline{U}$ such that
\[
U\cap V=\Red_U^{-1}(Z)
\]
\item If $U\cap V\neq\emptyset$, then the induced map between reductions
\[
\overline{U\cap V}\to \overline{U}
\]
is injective.
\end{enumerate}

\begin{exa}
This is taken from \cite[Ex.\ 4.8.5]{FP2004}.
Let $U_0=\mathset{\absolute{x}_K\le\absolute{\rho_0}_K}$ and
$U_1=\mathset{\absolute{\rho_2}_K\le\absolute{x}_K\le 1}$
with $\rho_0,\rho_1\in K^\times$ such that
$\absolute{\rho_1}_K\le\absolute{\rho_0}_K\le 1$.
The covering $\mathcal{U}=\mathset{U_0,U_1}$ of the unit disc is pure if and only if $\absolute{\rho_0}_K=\absolute{\rho_1}_K$.
\end{exa}

Given a pure affinoid covering $\mathcal{U}$
of a Mumford curve $X$, one obtaines the reduction of $X$ with respect to $\mathcal{U}$ by glueing the reductions of the patches $U\in\mathcal{U}$ along the reductions of their intersections,
and obtains a singular projective curve $\overline{X_{\mathcal{U}}}$ over $k$ whose irreducible components are all rational curves, and whose singular points are ordinary double points.

\begin{dfn}
The reduction graph  of a Mumford curve $X$ with respect to a pure affinoid covering $\mathcal{U}$ is the intersection graph of $Y=\overline{X_{\mathcal{U}}}$, i.e.\ the vertices correspond to the irreducible components of $Y$, and the edges to the double points.
\end{dfn}

It is a well-known fact that the reduction graph of
a Mumford curve of genus $g\ge0$ has first Betti number equal to $g$.
Its universal covering tree can be embedded as a subtree into
the Bruhat-Tits tree of $K$ by
 taking a maximal subtree $T$ of $G_{\mathcal{U}}$ and embedding it into $\mathcal{T}_K$. By replacing in $T$ all edges with geodetic paths as in Section \ref{sec:Red}, one obtains a verticial covering of a fundamental domain $\mathscr{F}$ for the action of
 the fundamental group $\Gamma$ of the graph, and by taking $\Gamma$-orbits, a verticial covering of $X(K)$.
 The corresponding quotient graph will be denoted as 
 $G_{\mathcal{U}}$ where this time $\mathcal{U}$ is the unique verticial covering of $X(K)$.
 
 \smallskip
 It is a well-known fact that the group $\Gamma$ is a free group of rank $g$, and that the construction above yields a faithful representation 
 \[
 \Gamma\to\PGL_2(K)
 \]
 as a discrete subgroup generated by $g$ hyperbolic (aka loxodromic) M\"obius transformations. More will be said in the Section about heat equations on Mumford curves.

\section{A $\pi$-adic dictionary}

The aim of this section is to generalise the dictionary of
\cite[\S 2]{pWaveGraph} to the setting of this article.

\smallskip
Let $Z\subset K$ be a compact measurable set, and $\mathcal{U}$ a
finite covering of $Z$ by disjoint sets of finite measure.
The space of continuous complex-valued functions on $Z$ will be denoted as $C(Z,\mathds{C}$ or simply as $C(Z)$. The space of bounded linear operators on a Banach space $F$ is denoted as $\mathcal{B}(F)$. 

\smallskip
Let $n=\absolute{\mathcal{U}}$, and let $n\times n$-matrices be indexed by $\mathcal{U}$. We call these $\mathcal{U}$-matrices.
If $A=(A_{UV})$ is a $\mathcal{U}$-matrix, then
we define
\[
\norm{A}_{\mathcal{U}}
=\sqrt{\sum\limits_{U,V\in\mathcal{U}} \absolute{A_{UV}}^2\mu(U)\mu(V)}
\]
this defines a norm on the algebra of $\mathcal{U}$-matrices, called $\mathcal{U}$-norm.

\begin{lem}[Dictionary]\label{dict}
There is an injective isometric homomorphism between algebras
\begin{align*}
b\colon\mathds{C}^{n\times n}&\to\mathcal{B}(C(Z,\mathds{C}))
\\
A=(A_{UV})&\mapsto A(x,y)=\sum\limits_{U,V\in\mathcal{U}}A_{UV}\Omega(x\in U)\Omega(y\in V)
\end{align*}
where the first space has the $\mathcal{U}$-norm, and the second space the norm
\[
\norm{\mathcal{A}}:=\sqrt{\int_Z\int_Z\absolute{A(x,y)}^2\,dy\,dx}
\]
where $\mathcal{A}\in\mathcal{B}(C(Z))$ has kernel $A(x,y)$.
\end{lem}

\begin{proof}
First observe that $b(A)$ is indeed a bounded linear operator:
\begin{align*}
\norm{b(A)u}_2^2&=
\int_Z\absolute{\int_ZA(x,y)u(y)\,dy}^2\,dx
\\
&\le\sum\limits_{U,V\in\mathcal{U}}\absolute{A_{UV}}^2
\int_Z\Omega(x\in U)\int_V\absolute{u(y)}^2\,dy\,dx
\\
&\le\sum\limits_{U,V\in\mathcal{U}}\absolute{A_{UV}}^2\mu(U)\mu(V)\norm{u}_2^2
\\
&=\norm{A}_{\mathcal{U}}^2\norm{u}_2^2
\end{align*}

Now, linearity and multiplicativity of $b$ are straightforward calculations. In the case of mutliplication, one may consult the corresponding part in the proof of \cite[Prop.\ 2.2]{pWaveGraph}.

\smallskip
Isometry follows from the observation that
\[
\norm{A}_{\mathcal{U}}^2=\int_Z\int_Z\absolute{A(x,y)}^2\,dy\,dx
=\norm{\mathcal{A}}^2
\]
Thus the assertion is proven.
\end{proof}

A $\mathcal{U}$-vector is an $n$-tupel with entries in $\mathcal{C}$ indexed
by $\mathcal{U}$. If $e=(e_V)$ is an $\mathcal{U}$-vector, then
there is an associated function $e(x)\in C(Z)$ defined as:
\[
e(x)=\sum\limits_{V\in\mathcal{U}}\mu(V)^{-1}e_V\,\Omega(x\in V)
\]
There is a map
\[
b\colon \mathds{C}^n\to C(Z),\;
e=(e_v)\mapsto e(x)
\]
We are convinced that the same letter $b$ as in Lemma \ref{dict}
will not cause confusion.

\smallskip
Now, there is a product defining an action of 
$\mathcal{B}(C(Z))$ on $C(Z)$. Namely, 
\[
u\mapsto \mathcal{A}u(x)=\int_ZA(x,y)u(y)\,dy
\]
We will write this in the usual way as $\mathcal{A}u$.

\begin{lem}\label{nice}
Let $A=(A_{UV})$ be a $\mathcal{U}$-matrix, and $e=(e_V)$ a $\mathcal{U}$-vector. Then
\[
b(A)b(e)=b(Ae)
\]
\end{lem}

\begin{proof}
We have 
\[
Ae=c=(c_U)
\]
with 
\[
c_U=\sum\limits_{V\in\mathcal{U}}A_{UV}e_V
\]
Hence, $b(c)$ is the function
\[
c(x)=\sum\limits_{U\in\mathcal{U}}
\sum\limits_{V\in\mathcal{U}}A_{UV}e_V\Omega(x\in U)
\]
On the other hand, $b(A)b(e)$ is the function
\begin{align*}
\int_ZA(x,y)e(y)\,dy&
=\sum\limits_{U,V\in\mathcal{U}}A_{UV}\mu(V)^{-1}e_V\int_V\Omega(y\in V)\,dy\,\Omega(x\in U)
\\
&=\sum\limits_{U,V\in\mathcal{U}}A_{UV}e_V\Omega(x\in U)
=b(Ae)
\end{align*}
This proves the assertion.
\end{proof}

\begin{cor}\label{eigenCorrespondence}
If $e=(e_V)$ is an eigenvector of $\mathcal{U}$-matrix $A$ with eigenvalue $\lambda\in\mathds{C}$, then
$b(e)$ is an eigenfunction of $b(A)$ for the same eigenvalue.
\end{cor}

\begin{proof}
We have
\[
b(A)b(e)=b(Ae)=b(\lambda e)=\lambda b(e)
\]
where the first equality is due to Lemma \ref{nice}.
\end{proof}

\section{Heat equations associated with affinoid coverings}

In the first subsection, we generalise the Kozyrev wavelets \cite{Kozyrev2002}
not only to $K$, but to have supports which need not be balls.
These turn out to be eigenfunctions of integral operators
constructed with the help of coverings of a compact clopen,
and symmetric graph adjacency matrices.
In the second subsection, we show that these integral  operators generate Feller semigroups, and
solve the Cauchy problem for the corresponding heat equation
on compact clopen sets.

\subsection{$K$-Wavelets}\label{sec:wavelets}

Let $Z$ be a measurable subset of $K$, and let
$\Cov(Z)$ be a Grothendieck topology on $X$
consisting of  coverings of $Z$ by mutually disjoint subsets.
The elements of $\Cov(Z)$ are called admissible 
coverings of $Z$. $\Cov(Z)$ is a partially order set with refinement $\le$ as the partial order, where we denote
\[
\mathcal{V}\le\mathcal{U}
\]
to mean that $\mathcal{V}$ is a refinement of $\mathcal{U}$.
A finite chain in $\Cov(Z)$ is a sequence
\[
\mathcal{U}_n<\dots<\mathcal{U}_i<\mathcal{U}_{i-1}<\dots<\mathcal{U}_0
\]
of coverings of $Z$. An infinite chain in $\Cov(Z)$ may or may not be bounded from above or below.
A chain is \emph{maximal}, if it cannot be made longer by inserting or appending new coverings to it.

\smallskip
We say that a refinement $\mathcal{V}<\mathcal{U}$ is \emph{strong}, if
\[
\mathcal{U}\cap\mathcal{V}=\emptyset
\]
A \emph{strong maximal chain} is a chain in which each refinement is strong, and which cannot be strongly refined. 

\smallskip
In the following, we fix a lift
\[
\tau\colon k\to O_K
\]
of the residue field $k$ into a set of representatives modulo the uniformiser $\pi$.

\smallskip
In the following, we assume that 
$Z$ is compact open, and 
our Grothendieck topology consists of coverings of $Z$ by rational affinoid sets (i.e.\ holed discs).

\begin{dfn}
Let $\mathcal{C}$ be a strong maximal chain of coverings in $\Cov(Z)$.
A \emph{$\mathcal{C}$-wavelet} on $Z$ is a family of functions
\[
\psi_{j,U}^{\mathcal{C}}=\mu(U)^{\frac12}\chi_K(p^{\frac{d-1}{e}}\tau(j)x)
\Omega(x\in U)
\]
with $U\in\mathcal{U}\in\mathcal{C}$,  $j\in k$,
and $p^{-\frac{d}{e}}$ the radius of the smallest hole in $U$.
\end{dfn}

\begin{exa}
Let $Z=\mathds{Q}_p$, and $\mathcal{C}$ the strong maximal chain of coverings in which each covering is given by the translates of
a fixed ball in $\mathds{Z}_p$. Then the $\mathcal{C}$-wavelets on $\mathds{Q}_p$ are the usual Kozyrev wavelets, if the $p$-adic lift $\tau$ from $\mathds{F}_p$ takes each element to its smallest natural counterpart.
\end{exa}

\begin{lem}\label{waveletInt1}
Let $\psi_{j,U}^{\mathcal{C}}$ be a $\mathcal{C}$-wavelet.
Then
\[
\int_Z\psi_{j,U}^{\mathcal{C}}(x)\,dx=0
\]
\end{lem}

\begin{proof}
This follows from the fact that
\[
\int_U\chi_K(ax)\,dx=0
\]
for $\absolute{a}_K$ sufficiently large.
\end{proof}

\begin{lem}\label{waveletInt2}
Let $\psi_{j,W}^\mathcal{C}$ be a $\mathcal{C}$-wavelet.
Then
\[
\int_W\absolute{x-y}_K^\alpha\psi_{j,W}^{\mathcal{C}}(y)\,dy=0
\]
for $x\in Z$.
\end{lem}

\begin{proof}
If for $x\in Z$ fixed, $\absolute{x-y}_K$ is constant for $y\in W$, then
the assertion follows from Lemma \ref{waveletInt1}.
In the alternative case,
we have
\[
\int_W\absolute{x-y}_K^\alpha\psi_{j,W}^{\mathcal{C}}(y)\,dy
=\sum\limits_{\nu\in\mathds{Z}}
p^{\nu}\int_{W\cap S_\nu}\psi_{j,W}^{\mathcal{C}}(y)\,dy=0
\]
where $S_\nu$ is the sphere centered in Zero with radius $p^\frac{\nu}{e}$.
The integral vanishes again by an argument similar to that in the proof of  Lemma \ref{waveletInt1}.
\end{proof}

\begin{dfn}
Let $\mathcal{U}\in\Cov(Z)$, and let
\[
A=(A_{UV})
\]
be a complex-valued matrix indexed by $U,V\in\mathcal{U}$, called a
\emph{$\mathcal{U}$-matrix}. The 
\emph{Z\'{u}\~{n}iga operator} associated with $A$ and $\mathcal{U}$ is
the $\pi$-adic Laplace operator $\mathcal{D}_{A,\mathcal{U}}^\alpha$
on $L^2(Z)$ given as
\[
\mathcal{D}_{A,\mathcal{U}}^\alpha f(x)
=\int_ZA_{\mathcal{U}}^\alpha(x,y)(f(y)-f(x))\,dy
\]
associated with the kernel
\[
A^\alpha_{\mathcal{U}}(x,y)=\sum\limits_{U,V\in\mathcal{U}}
A_{UV}\absolute{x-y}_K^{\alpha}\Omega(x\in U)\Omega(y\in V)
\]
with $\alpha\in\mathds{C}$.
\end{dfn}

For $Z=K$ and $A=\mathds{1}$ having all entries equal to $1$, this operator distinguishes only notationally form the well-known Vladimirov-Taibleson operator.

\smallskip
We assume in the following that
$\mathcal{U}\in\Cov(Z)$ is a fixed covering consisting of 
intersections of spheres of equal radius. Such a covering is called \emph{verticial}. And we assume that $A_{UU}=0$.

\smallskip
Thus, we define the \emph{degree} of $U\in\mathcal{U}\in\mathcal{C}$ as
\[
\deg^{Z}_{A,\alpha}(U)
=\sum\limits_{V\in\mathcal{U}}A_{UV}\int_V\absolute{x-y}^\alpha\,dy
\]
with $x\in U$.
Notice that this does not depend on $x\in U$:

\begin{lem}\label{absConst}
Let $\mathcal{U}$ be a verticial covering of $Z$, and
$U\neq V\in\mathcal{U}$. Then
\[
\absolute{x-y}_K=\text{const.}
\]
on $U\times V$.
\end{lem}

\begin{proof}
In $U,V$ we have equations 
\[
\absolute{x-u}_K=r,\quad\absolute{y-v}_K=s
\]
with $u\in U$, $v\in V$, respectively, then we have
\[
r,s<\absolute{u-v}_K
\]
So, we have
\[
\absolute{x-y}_K=\absolute{x-u+u-y}_K
=\absolute{u-y}_K=\absolute{u-v+v-y}_K
=\absolute{u-v}_K
\]
which is constant.
\end{proof}

Let
\[
C_{UV}^\alpha=\absolute{x-y}_K^\alpha
\]
for $x,y\in U\times V$, if $U\neq V$, and
\[
C_{UU}=0
\]
for $U,V\in\mathcal{U}$, and $\mathcal{U}$ a verticial covering of $Z$.

\smallskip
The following Lemma and Corollary are in line with 
\cite[Lem.\ 10.1]{Zuniga2020}.

\begin{lem}\label{boundedTest}
Assume that $A=(A_{UV})$ is a  matrix with non-negative real entries and
$A_{UU}=0$, and that $\alpha\in\mathds{R}$.
Then it holds true that
\[
\norm{\mathcal{D}_{A,\mathcal{U}}^\alpha f}_2
\le 2\sqrt{\sum\limits_{U\in\mathcal{U}}\deg_{A,\alpha}^{Z}(U)^2\mu(U)}\cdot\norm{f}_2
\]
for $f\in\mathcal{D}(Z)$.
\end{lem}

\begin{proof}
We have
\[
\norm{\mathcal{D}_{A,\mathcal{U}}^\alpha f}_2
\le\norm{\int_Z A_{\mathcal{U}}^\alpha(x,y)f(y)\,dy}_2
+\norm{\int_Z A_{\mathcal{U}}^\alpha(x,y)\,dy\,f(x)}_2
=:I_1+I_2
\]
and
\begin{align*}
I_1^2&=\int_Z\absolute{\int_Z A_{\mathcal{U}}^\alpha(x,y)f(y)\,dy}^2
\,dx
\\
&\le\int_Z\left(\int_Z\absolute{A_{\mathcal{U}}^\alpha(x,y)}
\absolute{f(y)}\,dy\right)^2\,dx
\\
&=\int_Z\left(\sum\limits_{U,V\in\mathcal{U}}A_{UV}\int_V\absolute{x-y}_K^\alpha \absolute{f(y)}\,dy\Omega(x\in U)\right)^2\,dx
\\
&\le\int_Z\left(\norm{f}_2\sum\limits_{U,V\in\mathcal{U}}A_{UV}C_{UV}^\alpha\mu(V)\Omega(x\in U)\right)^2\,dx
\\
&=\norm{f}_2^2\sum\limits_{U\in \mathcal{U}}
\deg_{A,\alpha}^{Z}(U)^2\int_Z\Omega(x\in U)\,dx
\\
&=\norm{f}_2^2\sum\limits_{U\in\mathcal{U}}\deg_{A,\alpha}^{Z}(U)^2\mu(U)
\end{align*}
as well as
\begin{align*}
I_2^2&=\int_Z\absolute{f(x)}^2
\left(\int_ZA_{\mathcal{U}}^\alpha(x,y)\,dy\right)^2\,dx
\\
&=\int_Z\absolute{f(x)}^2
\left(\sum\limits_{U\in\mathcal{U}}\deg_{A,\alpha}^{Z}(U)\Omega(x\in U)\right)^2\,dx
\\
&=\sum\limits_{U\in\mathcal{U}}\deg_{A,\alpha}^{Z}(U)^2\int_Z\absolute{f(x)}^2\Omega(x\in U)\,dx
\\
&\le\norm{f}_2^2\sum\limits_{U\in\mathcal{U}}\deg_{A,\alpha}^{Z}(U)^2\mu(U)
\end{align*}
From this, the assertion follows.
\end{proof}

\begin{cor}\label{bounded}
Under the assumptions of Lemma \ref{boundedTest},
it holds true that $\mathcal{D}_{A,\mathcal{U}}^\alpha$ is a bounded linear operator on $L^2(Z)$.
\end{cor}

\begin{proof}
According to Lemma  \ref{boundedTest},
$\mathcal{D}_{A,\mathcal{U}}^\alpha$ is bounded on the space of test functions. As the Schwartz space $\mathcal{D}(Z)$ is dense in $L^2(Z)$, it follows that the operator $\mathcal{D}_{A,\mathcal{U}}^\alpha$ is bounded on $L^2(Z)$. 
\end{proof}

\begin{lem}\label{self-adjoint}
If $A$ is symmetric with $A_{UU}=0$, and $\alpha\in\mathds{R}$, then $\mathcal{D}_{A,\mathcal{U}}^\alpha$ is self-adjoint.
\end{lem}

\begin{proof}
A simple calculation shows that
\[
\langle f,\mathcal{A}^\alpha_{\mathcal{U}} g\rangle
=\langle\mathcal{A}_{\mathcal{U}}^\alpha f,g\rangle 
\]
From this, it follows that $\mathcal{D}_{A,\mathcal{U}}^\alpha$, whose kernel is
\[
A_{\mathcal{U}}^\alpha(x,y)-
\sum\limits_{U,V\in\mathcal{U}}\deg_{A,\alpha}^{Z}(U)\mu(U)
\delta(x-y)\Omega(x\in U)\Omega(x\in V)
\]
is self-adjoint.
\end{proof}

In particular, if the conditions of Lemma \ref{self-adjoint} 
and Lemma \ref{boundedTest}
are satisfied, then $\mathcal{D}_{A,\mathcal{U}}^\alpha$ is diagonalisable by a unitary transformation. The following Theorem makes this explicit.

\begin{thm}\label{TVEigen}
Let $\alpha\in\mathds{R}$.
Let $\mathcal{U}\in\Cov(Z)$ be verticial,
and let $\mathcal{U}$-matrix $A$
be non-negative and symmetric such that $A_{UU}=0$.
Let $\mathcal{C}$ be a strong chain of admissible coverings of $Z$ such that $\mathcal{U}\in\mathcal{C}$ is on top, and such that $\mathcal{C}$ is maximal with this property.
Then the $\mathcal{C}$-wavelets together with the
functions
\[
e(x)=\sum\limits_{V\in\mathcal{U}}\mu(V)^{-1}e_V\,
\Omega(x\in V)
\]
where $e=(e_V)$ is an eigenvector of the 
Laplacian whose adjacency matrix is
matrix
\[
B^\alpha = (B_{UV}),\quad
B_{UV}=\begin{cases}
0,&U=V\\
A_{UV}\absolute{x-y}_K^\alpha,&\text{otherwise}
\end{cases}
\]
with $x\in U$, $y\in V$,
such that $\norm{e(x)}_2=1$,
are an orthonormal basis of 
$L^2(Z)$ consisting of eigenfunctions of $\mathcal{D}^\alpha_{A,\mathcal{U}}$.
\end{thm}

\begin{proof}
First, the assertion about the wavelets.
Let $W\subseteq U\in\mathcal{U}$. Then
\begin{align*}
\int_Z A_{\mathcal{U}}^\alpha(x,y)\psi_{j,W}^{\mathcal{C}}(y)\,dy
=\sum\limits_{V\in\mathcal{U}}A_{UV}\int_W\absolute{x-y}_K^\alpha\psi_{j,W}^{\mathcal{C}}(y)\,dy=0
\end{align*}
where the latter equality holds true because of Lemma \ref{waveletInt2}.
And now,
\begin{align*}
\int_ZA_{\mathcal{U}}^\alpha(x,y)\,dy\,\psi_{j,W}^{\mathcal{C}}(x)
&=\sum\limits_{V\in\mathcal{U}}A_{UV}
\int_V\absolute{x-y}_K^\alpha\,dy\,\psi_{j,W}^{\mathcal{C}}(x)
\\
&=\deg_{A,\alpha}^{Z}(U)\cdot\psi_{j,W}^{\mathcal{C}}(x)
\end{align*}
Together, this shows that $\psi_{j,W}^{\mathcal{C}}$ with $W\subseteq U\in\mathcal{U}$ is an eigenfunction of $\mathcal{D}_{A,\mathcal{U}}^\alpha$ with eigenvalue
$-\deg_{A,\alpha}^{Z}(U)$.

\smallskip
Secondly, the assertion about the eigenvectors of $B^\alpha$.
As, by the dictionary of Lemma \ref{dict}, the operator $\mathcal{D}_{A,\mathcal{U}}^\alpha$ corresponds to the Laplacian of the
matrix 
$B^\alpha$ which by Lemma \ref{absConst} does not depend on the choice of $x,y\in U\times V$, it follows that $e(x)$ is an eigenfunction of $\mathcal{D}_{A,\mathcal{U}}^\alpha$, if and only if $e$ is an eigenvector of $B^\alpha$, cf.\ Corollary \ref{eigenCorrespondence}.

\smallskip
Now, the $\mathcal{C}$-wavelets $\psi_{j,W}$ form an orthonormal basis of the closed subspace 
\[
L_{\mathcal{U}}^2(Z)=
\mathset{f\in L^2(Z)\colon \forall\;U\in\mathcal{U}\colon
\int_Uf(x)\,dx=0}
\]
of $L^2(Z)$.
Denote the closed subspace of $L^2(Z)$ generated by the functions $e(x)$ with $L_A^2(Z)$, then by Lemma \ref{decompU} below,
we have an orthogonal decomposition
\[
L^2(Z)=L_{\mathcal{U}}^2(Z)\oplus L_A^2(Z)
\]
This proves the assertion about having an orthonormal basis of $L^2(Z)$.
\end{proof}

\begin{lem}\label{decompU}
There is an orthogonal decomposition
\[
L^2(Z)=L^2_{\mathcal{U}}(Z)\oplus L_A^2(Z)
\]
\end{lem}

\begin{proof}
Let $E$ be the orthogonal complement of $L_{\mathcal{U}}^2(Z)$ in $L^2(Z)$.

\smallskip
Observe that 
\[
L_A^2(Z)=\mathset{f\in L^2(Z)\colon \forall\;U\in\mathcal{U}\colon f|_U=\text{const.}}\subseteq E
\]
Namely, the first equality is clear by definition of $L_A^2(Z)$. Now,
observe that if $f$ is constant on each $U\in\mathcal{U}$, then $f$ is orthogonal to each $\mathcal{C}$-wavelet $\psi_{j,W}$;
in the case $W\subset U$, use Lemma \ref{waveletInt1}.

\smallskip
On the other hand, let $f\in E$. 
Then 
\[
h=\sum\limits_{U\in\mathcal{U}}\mu(U)^{-1}\int_U f(x)\,dx\,\Omega(x\in U)
\in L^2(Z)
\]
is constant on each $U$, because $f\in L^1(Z)$. 
Now,
\[
g=f-h\in L^2_{\mathcal{U}}(Z)
\]
Then it must be that $g=0$.
Hence, $f=h$, and this function is 
in $L_A^2(Z)$.
This proves the asserted orthogonal decomposition.
\end{proof}

\begin{exa}\label{exaNonComp}
Let $K=\mathds{Q}_p$, $U=p\mathds{Z}_p$, $V=1+p\mathds{Z}_p$, and $Z=U\cup V$ with the covering $\mathcal{U}=\mathset{U,V}$. Let
\[
A=\begin{pmatrix}
0&1\\1&0
\end{pmatrix}
\]
We have
\[
A_{\mathcal{U}}^\alpha
=\absolute{x-y}_K^\alpha\Omega(x\in U)\Omega(y\in V)
+\absolute{x-y}_K^\alpha\Omega(x\in V)\Omega(y\in U)
\]
Let $W_n=p^{2n}\mathds{Z}_p$ with $n\ge 1$. Then the
wavelet
\[
\psi_{W_n}(x)=p^{n}\chi(p^{-2n-1}x)\Omega(x\in W_n)
\]
is an eigenfunction of $\mathcal{D}_{\mathcal{U}}^\alpha$. Namely, we have
\begin{align*}
\mathcal{A}_{\mathcal{U}}^\alpha\psi_{W_n}(x)
&=\int_{W_n}\absolute{x-y}_K^\alpha \psi_{W_n}(y)\,dy\,\Omega(x\in V)
\\
&=C_{UV}^\alpha\int_{W_n}\psi_{W_n}(y)\,dy\,\Omega(x\in V)
=0
\end{align*}
because the integral vanishes by Lemma \ref{waveletInt1}.
And we have
\begin{align*}
\int_ZA_{\mathcal{U}}^\alpha(x,y)&\,dy\,\psi_{W_n}(x)
=\int_V\absolute{x-y}_K^\alpha\,dy\,\psi_{W_n}(x)
\\
&=\mu(V)\psi_{W_n}(x)
\end{align*}
Hence, $\psi_{W_n}(x)$ with $n\ge 1$ is indeed an eigenfunction for eigenvalue
$-\mu(V)$.  
\end{exa}

Example \ref{exaNonComp} already shows that, in general, $\mathcal{D}_{A,\mathcal{U}}^\alpha$ cannot be expected to be a compact operator on $L^2(Z)$, not even for $\alpha=0$. Theorem \ref{TVEigen} only confirms this.

\smallskip
We will now further decompose $L_A^2(Z)$ by using the graph structure on $\mathcal{U}$ given by the $\mathcal{U}$-matrix $A$. Let $G$ be the graph whose vertex set is $\mathcal{U}$, and such that $A$ is an adjacency matrix of $G$. Define the following $\mathds{C}$-vector spaces:
\begin{align*}
A(G,\mathds{C})&:=\mathset{\beta\colon E(G)\to\mathds{C}}
\\
H(G,\mathds{C})&=\mathset{f\colon V(G)\to\mathds{C}}
\end{align*}
where $V(G)=\mathcal{U}$ and $E(G)$ are the vertex and edge sets of $G$, respectively. From graph theory, there is a well-known exact sequence
\begin{align}\label{exactSeq}
\xymatrix{
0\ar[r]& c(G,\mathds{C})\ar[r]& A(G,\mathds{C})\ar[r]^d
& H(G,\mathds{C})\ar[r]^\phi &\mathds{C}^{b_0(G)}\ar[r]& 0
}
\end{align}
where $b_0(G)$ is the $0$-th Betti number of $G$, and
$c(G,\mathds{C})$ is the kernel of the linear map $d$.
The map $d$ is defined as follows:
\[
d(\beta)\colon v\mapsto\sum\limits_{e\in E(G)\atop e\dashv v}\beta(e)
\]
where $e\dashv v$ means that edge $e$ is attached to vertex $v$.
The map $\phi$ is defined as follows:
\[
f\mapsto \left(\sum_{v\in C)}f(v)\right)_{C\in C(G)}
\]
where $C(G)$ is the set of all connected components of $G$.

\smallskip
Furthermore, there is a linear map
\[
\sigma\colon L^2(Z)\to H(G,\mathds{C}),
f\mapsto\left(U\mapsto \int_Uf(x)\,dx\right)
\]
\begin{lem}\label{immediateL2s}
The map $\sigma$ is surjective, and
it holds true that
\[
\ker(\phi\circ\sigma)=
L_{C(G)}^2(Z)
\]
where 
\[
L^2_{C(G)}(Z)=\mathset{f\in L^2(Z)\colon \forall C\in C(G)\colon \int_{C}f(x)\,dx=0}
\]
and, further,
\[
\ker(\sigma) = L^2_{\mathcal{U}}(Z)
\]
\end{lem}

\begin{proof}
This is immediate.
\end{proof}

\begin{cor}
It holds true that
\[
L^2_{C(G)}(Z)/L^2_{\mathcal{U}}(Z)\cong d(A(G,\mathds{C}))
\]
and
\[
L^2(Z)/L^2_{\mathcal{U}}(Z)\cong H(G,\mathds{C})
\]
In particular, it holds true that
\begin{align*}
L^2(Z) &\cong L_{\mathcal{U}}\oplus d(A(G,\mathds{C}))\oplus \mathds{C}^{b_0(G)}
\\
&\cong L_{\mathcal{U}}\oplus A(G,\mathds{C})/c(G,\mathds{C})\oplus 
\mathds{C}^{b_0(G)}
\end{align*}
\end{cor}

\begin{proof}
This follows from the exact sequence (\ref{exactSeq}) and from 
Lemma \ref{immediateL2s}. In particular, there is a decomposition
\[
H(G,\mathds{C})\cong A(G,\mathds{C})/c(G,\mathds{C})\oplus \mathds{C}^{b_0(G)}
\]
from the exact sequence.
\end{proof}

\begin{rem}
Notice that the direct summand $\mathds{C}^{b_0(G)}$ equals the eigenspace of $\mathcal{D}_{\mathcal{U}}^\alpha$ for eigenvalue zero.
\end{rem}

\subsection{Cauchy problem for the heat equation on $Z$}

Let $Z\subset K$ be a compact clopen set, and $\mathcal{U}$ a verticial covering of $Z$ by rational affinoids.

\smallskip
Following
\cite[Ch.\ 4.2]{EK1986}, we consider the Banach space
$C(Z)$ of continuous real-valued functions on
$Z$ w.r.t.\ the $L_\infty$-norm $\norm{\cdot}_\infty$.

\smallskip
An operator is said to  be \emph{positive}, if it maps non-negative functions to  non-negative functions.

\smallskip
A semigroup $T(t)$ on $C(Z)$ is \emph{positive},
if for each $t\ge 0$ the operator $T(t)$ is positive.

\smallskip
An operator on $C(Z)$ is said to satisfy the \emph{positive maximum principle}, if for all 
$f\in\dom(A)$, $x_0\in\Omega$, we have
\[
\sup\limits_{x\in\Omega}f(x)=f(x_0)\ge0\quad
\Rightarrow
\quad
Af(x_0)\le 0
\]

\begin{thm}[Yosida-Hille]\label{Yosida-Hille}
The closure $\bar{\mathcal{A}}$ of a linear operator $\mathcal{A}$ on $C(Z)$
with doubly invariant kernel 
is single-valued and generates a strongly continuous positive, contraction semigroup  (aka \emph{Feller semigroup}) $T(t)$ on $C(Z)$, if and only if
\begin{enumerate}
\item $\dom(\mathcal{A})$ is dense in $C(Z)$.
\item $\mathcal{A}$ satisfies the positive maximum principle.
\item The resolvent set $\Res(\lambda-\mathcal{A})$ is dense in $C(Z)$ for some $\lambda>0$. 
\end{enumerate}
\end{thm}

\begin{proof}
Cf.\ \cite[Thm.\ 4.2.2]{EK1986}.
\end{proof}

We take the following definition from \cite{Zuniga2020} in our situation:

\begin{dfn}
A family of bounded linear operators $P_t \colon C(Z) \to C (Z)$ is called a \emph{Feller semigroup},
if
\begin{enumerate}
\item $P_{s+t}=P_sP_t$, $P_0=I$ (semigroup).
\item for any $h\in C(Z)$, it holds true that $\lim\limits_{t\to 0}\norm{P_th-h}_\infty=0$ (strongly continuous).
\item For $h\in C(Z)$ and $t\ge0$, we have: $0\le h\le 1$ $\Rightarrow$ $0\le P_th\le 1$ (positive contraction).
\end{enumerate}
\end{dfn}

Thus, the criteria of Yosida-Hille (Theorem \ref{Yosida-Hille})
imply that $\mathcal{A}$ has a closed extension which generates a Feller semigroup.

\smallskip
Let $A$ be a symmetric  $\mathcal{U}$-matrix whose entries are non-negative, and
with $A_{UU}=0$.

\begin{lem}
Let $\epsilon>0$. Then $\epsilon\mathcal{D}_{A,\mathcal{U}}^\alpha$ with $\alpha\in\mathds{R}$ generates a 
Feller 
semigroup $\exp\left(t\epsilon\mathcal{D}_{A,\mathcal{U}}^\alpha\right)$ on
$C(Z)$.
\end{lem}

\begin{proof}
First observe that by Lemma \ref{bounded}, $\mathcal{D}_{A,\mathcal{U}}^\alpha$ is bounded, hence $\epsilon \mathcal{D}_{A,\mathcal{U}}^\alpha$ is a closed linear operator.

\smallskip
We now go through the steps in the proof of \cite[Lemma 4.1]{Zuniga2020}.

\smallskip
The conditions 1.\ and 2.\ of Theorem \ref{Yosida-Hille} are straightforward. Condition 3.\ means that there exists $\eta>0$ such that for any $h\in C(Z)$ the equation
\begin{align}\label{eigenequation}
\left(\eta-\epsilon\mathcal{D}_{A,\mathcal{U}}^\alpha\right)u=h
\end{align}
has a solution $u\in C(Z)$. Let
\[
g(x)=\int_{Z}A^\alpha_{\mathcal{U}}(x,y)\,dy
\]
Observe that
$g\in C(Z)$.
Hence, $g$ is bounded, because $Z$ is compact.
This means that
there exists some $C>0$ such that
\[
\norm{g}_\infty\le C
\]
Now, rewrite equation (\ref{eigenequation}) in this way:
\[
u(x)-\epsilon\int_{Z}u(y)
\frac{A^\alpha_{\mathcal{U}}(x,y)}{\eta+\epsilon g(x)}\,dy
=\frac{h(x)}{\eta+\epsilon g(x)}
\]
Observe that the right hand side is in $C(Z)$.
Now, the operator
\[
T\colon C(Z)\to C(Z),\;
u(x)\mapsto \epsilon\int_{Z}u(y)\frac{A^\alpha_{\mathcal{U}}(x,y)}{\eta+\epsilon g(x)}\,dy
\]
satisfies
\[
\norm{T}\le\frac{\epsilon C}{\eta}
\]
By taking $\eta>\epsilon C$, the operator $I-T$ has an inverse in
$C(Z)$. And notice that $\eta$ is independent of $h$.

\smallskip
It follows that the closed operator $\mathcal{D}_{A,\mathcal{U}}^\alpha$ generates a semigroup $Q_t$ ($t\ge 0$) satisfying the conditions of Yosida-Hille (Theorem \ref{Yosida-Hille}). On the other hand,
since $\epsilon \mathcal{D}_{A,\mathcal{U}}^\alpha$ is a bounded linear operator on a Banach space, it follows that 
\[
\exp\left(t\epsilon\mathcal{D}_{A,\mathcal{U}}^\alpha\right),
\quad t\ge 0
\]
is a uniformly continuos semigroup, and by using that the infinitesimal generators of $\exp\left(t\epsilon\mathcal{D}_{A,\mathcal{U}}^\alpha\right)$ and $Q_t$  (with $t\ge0$) agree, we have
\[
Q_t=\exp\left(t\epsilon\mathcal{D}_{A,\mathcal{U}}^\alpha\right)
\]
for $t\ge 0$, cf.\ e.g.\ \cite[Thm.\ 1.2 and Thm.\ 1.3]{Pazy1983}
\end{proof}

We are now going to look at the following Cauchy problem:

\begin{task}[Cauchy Problem]\label{CPZ}
Find $h(t,x)\in C^1\left((0,\infty),C(Z)\right)$ such that
\begin{align}\label{heateqZ}
\left(\frac{\partial}{\partial t}-\epsilon\mathcal{D}_{A,\mathcal{U}}^\alpha\right) h(t,x)=0
\end{align}
for $t\ge0$, $x\in\Omega$ which satisfies the initial condition
\[
h(0,x)=h_0(x)
\]
where $h_0\in C(Z)$ is fixed.
\end{task}

As the corresponding semigroup is Feller, it describes a 
$\pi$-adic heat equation on 
$Z$. 
Consequently
there is a $\pi$-adic diﬀusion process in $Z$ attached to the diﬀerential equation (\ref{heateqZ}).

\begin{thm}
There exists a probability measure $p_t(x,\cdot)$ with $t\ge0$, $x\in\Omega$ on the Borel $\sigma$-algebra of $\Omega$ such that the Cauchy Problem (Task \ref{CPZ}) has a unique solution of the form
\[
h(t,x)=\int_{Z}h_0(y)p_t(x,dy)
\]
In addition, $p_t(x,\cdot)$ is the transition function of a Markov process whose paths are right continuous and have no discontinuities other than jumps.
\end{thm}

\begin{proof}
The proof of \cite[Thm.\ 4.2]{Zuniga2020} carries over word for word.
\end{proof}

In fact, the operator $\epsilon\mathcal{D}_{A,\mathcal{U}}^\alpha$ generates a convolution semigroup, as we will see below. First, set
\[
\Lambda_{A,\mathcal{U}}^\alpha(x,y)
=\sum\limits_{U,V\in\mathcal{U}}
A_{UV}\absolute{y}_K^\alpha\Omega(y\in V)\Omega(x\in U)
-\deg_{A,\mathcal{U}}^\alpha(U)\delta_{y,0}
\Omega(x\in U)
\]
\begin{lem}
It holds true that
\[
\mathcal{D}_{A,\mathcal{U}}^\alpha f
=\Lambda_{A,\mathcal{U}}^\alpha(\cdot,y)*_yf(y)
\]
for $f\in L^2(Z)$ and
$\alpha\in\mathds{R}$. In other words,
the operator $\mathcal{D}_{A,\mathcal{U}}^\alpha$ acts by convolution on $L^2(Z)$.
\end{lem}

\begin{proof}
This is a simple calculation.
\end{proof}

\begin{cor}
The solution of Task \ref{CPZ} 
is obtained by convolution with the function
\[
\exp\left(t\epsilon\Lambda_{A,\mathcal{U}}^\alpha(x,y)\right)
\]
in the variable $y$.
\end{cor}

\begin{proof}
This is immediate.
\end{proof}

\section{Heat equation on Mumford curves}

In this section, let $\Gamma$ be a Schottky group generated by $g\ge 1$ hyperbolic M\"obius transformations. Let $\Omega\subset\mathds{P}^1(K)$ be the complement of the closure of the set of limit points of $\Gamma$, and $X=\Omega/\Gamma$ the corresponding Mumford curve. Let $\mathscr{F}\subset\Omega$ be a fundamental domain for the action of $\Gamma$, which we choose to be a compact clopen subset of $K$. Let $\mathcal{U}\in\Cov(\mathscr{F})$ be verticial. 

\smallskip
The aim is to generalise the results of the previous section to the setting of $\Gamma$-invariant function spaces.
This yields heat equations on Mumford curves.

\subsection{Invariant wavelets over $K$}
Let $f\colon\mathscr{F}\to\mathds{C}$ be a function, which we can write as
\[
f(x)=\sum\limits_{U\in\mathcal{U}}f_U(x)\Omega(x\in U)
\]
Then we can define a $\Gamma$-periodic extension of $f$ via
\[
f_\Gamma(x)
=\sum\limits_{\gamma\in\Gamma}\sum\limits_{U\in\mathcal{U}}
f_U(\gamma^{-1}x)\Omega(x\in\gamma U)
\]
where $x\in\Omega$.

\smallskip
In general, such a function is not $L^2$-integrable.
That is why we define a weighted integral as follows:
Let $f_\Gamma$ be a $\Gamma$-periodic function. Then
\[
\int_\Omega^sf_\Gamma(x)\,dx:=
\sum\limits_{U\in\mathcal{U}}
\sum\limits_{\gamma\in\Gamma}
m_{\gamma,U}^{-s}\int_{\gamma U}f_U(\gamma^{-1}x)\,dx
\]
with
\[
m_{\gamma,U}=\max\mathset{1,\max\limits_U\absolute{\gamma'}_K}
\]
and $s>1$. 

\smallskip
The inner product for $\Gamma$-periodic functions $f_\Gamma,g_\Gamma$ thus becomes
\begin{align*}
\langle f_\Gamma,g_\Gamma\rangle_s
&=\int_{\Omega,s}f_\Gamma(x)\overline{g_\Gamma(x)}\,dx
\\
&=\sum\limits_{\gamma\in \Gamma}\sum\limits_{U\in\mathcal{U}}
m_{\gamma,U}^{-s}
\int_{\gamma U}f_U(\gamma^{-1}x)\overline{g_U(\gamma^{-1}y)}\,dx
\end{align*}

\begin{dfn}
We have the space
\[
L^2_s(\Omega)^\Gamma=\mathset{
f\colon\Omega\to\mathds{C}\colon \text{$f$ is $\Gamma$-periodic and 
$\norm{f}_s<\infty$}}
\]
where
\[
\norm{f}_s:=\sqrt{\langle f,f\rangle_s}
\]
is the norm associated with the inner product $\langle\cdot,\cdot\rangle_s$ for $s>1$.
The same space is also defined as the space
\[
L^2_s(X):=L^2_s(\Omega)^\Gamma
\]
of $L^2$-functions on the Mumford curve $X$ with parameter $s>1$.
\end{dfn}

Let $C(\Omega)^\Gamma$ be the vector subspace of $C(\Omega)$ consisting of $\Gamma$-invariant continuous functions.

\begin{lem}\label{absderivative}
There are only finitely many $\delta\in\Gamma$ such that
$\absolute{\delta'z}$ takes a fixed value on a given small clopen set in $K$.
\end{lem}

\begin{proof}
Observe that $\absolute{\delta}_K$ is a locally constant function. Hence, $\delta\in\Gamma$ is unique up to a constant multiple from $O_K^\times$.

The map $\delta\mapsto u\delta$ with $u\in O_K^\times$ 
gives an action of $O_K^\times$ on $\Gamma$.
The set
\[
(O_K^\times\, \delta)\cap\Gamma
\]
with $\delta\in \Gamma$, is finite, because $\Gamma$ is
a discrete subgroup of $\PGL_2(K)$.

Together, this means that there are only finitely many $\delta\in\Gamma$ such that $\absolute{\delta'}$ takes on a given value on a specified small clopen in $K$.
\end{proof}

\begin{lem}
The vectorspace $L^2_s(\Omega)^\Gamma$ is a Hilbert space, where $s>1$. 
\end{lem}

\begin{proof}
Let $f_n$ be a Cauchy sequence in $L^2_s(\Omega)^\Gamma$.
As for functions 
\[
g(x)=\sum\limits_{\gamma\in\Gamma}\sum\limits_{U\in\mathcal{U}}g_U(\gamma^{-1}x)\Omega(x\in\gamma U)
\]
we have
\[
\norm{g}_s^2=\sum\limits_{\gamma\in\Gamma}\sum\limits_{U\in\mathcal{U}}m_{\gamma,U}^{-s}\norm{g_U}_2^2
\]
it follows that each
$f_{n,U}$ is a Cauchy sequence in $L^2(U)$.
So, let $f_U\in L^2(U)$ be the limit of $f_{n,U}$, and let
$\epsilon>0$, and $n$ such that for all $U\in\mathcal{U}$ we have
\[
\norm{f_{n,U}-f_U}_2<\epsilon
\]
Define
\[
f(x)=\sum\limits_{\gamma\in \Gamma}\sum\limits_{U\in\mathcal{U}}
f_U(\gamma^{-1}x)\Omega(x\in \gamma U)
\]
and verify that
\[
\norm{f_n-f}_s^2
=\sum\limits_{\gamma\in\Gamma}\sum\limits_{U\in\mathcal{U}}
m_{\gamma,U}^{-s}\norm{f_{n,U}-f_U}_2^2
<\epsilon\cdot\sum\limits_{\gamma\in\Gamma}\sum\limits_{U\in\mathcal{U}}m_{\gamma,U}^{-s}
\]
By Lemma \ref{absderivative}, the latter series is convergent.
Hence, we can conclude that $f_n$ converges to $f$ in $L_s^2(\Omega)^\Gamma$. We have shown that $L_s^2(\Omega)^\Gamma$ is a complete inner product space, i.e.\ a Hilbert space.
\end{proof}

\begin{lem}
If $f\in L^2(\mathscr{F})$, then
$f_\Gamma\in L^2_s(\Omega)^\Gamma$.
\end{lem}

\begin{proof}
We have
\begin{align*}
\langle f_\Gamma,f_\Gamma\rangle_s
&=\sum\limits_{\gamma\in\Gamma}\sum\limits_{U\in\mathcal{U}}
m_{\gamma,U}^{-s}\int_{\gamma U}
\absolute{f_U(\gamma^{-1}x)}_K^2\,dx
\\
&=\sum\limits_{\gamma\in\Gamma}\sum\limits_{U\in\mathcal{U}}
m_{\gamma,U}^{-s}\int_U\absolute{f_U(z)}_K^2\absolute{\gamma'z}_K\,dz
\\
&\le\sum\limits_{\gamma\in\Gamma}
\sum\limits_{U\in\mathcal{U}}
m_{\gamma,U}^{-s}\norm{f_U}^2\int_U\absolute{\gamma'z}_K\,dz
\\
&=\sum\limits_{\gamma\in\Gamma}\sum\limits_{U\in\mathcal{U}}
\norm{f_U}^2\min\mathset{1,\int_U\absolute{\gamma'z}_K}^{1-s}
\end{align*}
This is, by Lemma \ref{absderivative} and $\norm{f_U}<\infty$, a convergent power series in the variable $p^{\frac{1-s}{e}}$, because $s>1$.
This proves the assertion.
\end{proof}

The group $\Gamma$ acts on $\Cov(\Omega)$ by the map
\[
\mathcal{U}\mapsto\gamma\mathcal{U}:=\mathset{\gamma U\colon U\in\mathcal{U}}
\]
Under this action, chains map to chains (possibly in reverse order), and strong maximal chains are taken to strong maximal chains.

\smallskip
Fix $\mathcal{U}\in\Cov(\mathscr{F})$.
We have the $\pi$-adic operator 
$\mathcal{A}_{\Gamma}^\alpha$ associated with 
the complex-valued $\mathcal{U}$-matrix
$A=(A_{UV})$
defined 
via the kernel:
\begin{align*}
A_{\Gamma}^\alpha(x,y)
&=\sum\limits_{\gamma,\delta\in\Gamma}
A_{\mathcal{U}}^\alpha(\gamma^{-1}x,\delta^{-1}y)
\\
&=\sum\limits_{U,V\in\mathcal{U}}A_{UV}
\sum\limits_{\gamma,\delta\in\Gamma}
\absolute{\gamma^{-1}x-\delta^{-1}y}_K^\alpha
\Omega(x\in\gamma U)\Omega(y\in\delta V)
\end{align*}
and the associated $\pi$-adic doubly $\Gamma$-invariant
Taibleson-Vladimirov operator as follows:
\[
\mathcal{D}_{A,\Gamma}^\alpha f(x)
=\int_{\Omega,s}\mathcal{A}_{A,\Gamma}^\alpha(f(y)-f(x))\,dy
\]
It can immediately be seen that $\mathcal{D}_{A,\Gamma}$ takes $\Gamma$-invariant functions to $\gamma$-invariant functions.

\smallskip
We define
\[
\deg^\alpha_{A,\Gamma}(U)
=\sum\limits_{V\in\mathcal{U}}
A_{UV}\sum\limits_{\delta\in\Gamma}
\int_{\delta V,s}\absolute{x-\delta^{-1}y}_K^\alpha\,dy
\]
for $x\in U$. 
Observe that, if $A_{UU}=0$ for all $U\in\mathcal{U}$, then
$\deg_{A,\Gamma}^\alpha(U)$ indeed only depends on $U$, not on $x\in U$.

\begin{lem}\label{Gammabounded}
Let $A$ be a real matrix with non-negative entries and $A_{UU}=0$ for all $U\in\mathcal{U}$, $\alpha\in\mathds{R}$, and $s>1$.
Then it holds true that
\[
\norm{\mathcal{D}_{A,\Gamma}^\alpha f}_s
\le 2
\sqrt{\sum\limits_{U\in\mathcal{U}}\deg_{A,\Gamma}^\alpha(U)^2
\mu(U)
\sum\limits_{\gamma\in\Gamma}m_{\gamma,U}^{-s}}
\cdot \norm{f}_s
\]
\end{lem}

\begin{proof}
We have
\begin{align*}
\norm{\mathcal{D}_{A,\Gamma}^\alpha f}_s
\le
\norm{\int_{\Omega,s}A_\Gamma^\alpha(x,y)f(y)\,dy}_s
+\norm{\int_{\Omega,s}A_\Gamma^\alpha(x,y)\,dy\,f(x)}_s
=:I_1+I_2
\end{align*}
and
\begin{align*}
I_1^2
&=\int_{\Omega,s}\absolute{\int_{\Omega,s}A_\Gamma^\alpha(x,y)f(y)\,dy}^2\,dx
\\
&\le
\int_{\Omega,s}
\left(\sum\limits_{U,V\in\mathcal{U}}A_{UV}
\sum\limits_{\gamma,\delta\in\Gamma}
\int_{\delta V,s}\absolute{\gamma^{-1}x-\delta^{-1}y}_K^\alpha
\absolute{f(y)}\,dy\,\Omega(x\in\gamma U)\right)^2\,dx
\\
&\le \int_{\Omega,s}
\left(\norm{f}_s\sum\limits_{U,V\in\mathcal{U}}
A_{UV}
\sum\limits_{\gamma,\delta\in\Gamma}
\int_{\delta V,s}\absolute{\gamma^{-1}x-\delta^{-1}y}^\alpha\,dy
\,\Omega(x\in\gamma U)\right)^2\,dx
\\
&=\norm{f}_s^2
\sum\limits_{U\in\mathcal{U}}\deg_{A,\Gamma}^\alpha(U)^2\mu(U)\sum\limits_{\gamma\in\Gamma}m_{\gamma,U}^{-s}
\end{align*}
as well as
\begin{align*}
I_2^2
&=\int_{\Omega,s}\absolute{f(x)}^2
\left(\int_{\Omega,s}A_\Gamma^\alpha(x,y)\,dy\right)^2\,dx
\\
&=\int_{\Omega,s}\absolute{f(x)}^2
\left(\sum\limits_{U\in\mathcal{U}}\deg_{A,\Gamma}^\alpha(U)
\sum\limits_{\gamma\in\Gamma}\Omega(x\in\gamma U)\right)^2\,dx
\\
&=\sum\limits_{U\in\mathcal{U}}
\deg_{A,\Gamma}^\alpha(U)^2
\sum\limits_{\gamma\in\Gamma}
\int_{\gamma U,s}\absolute{f(x)}^2\Omega(x\in\gamma U)\,dx
\\
&\le
\norm{f}_s^2\sum\limits_{U\in\mathcal{U}}
\deg_{A,\Gamma}^\alpha(U)^2
\mu(U)\sum\limits_{\gamma\in\Gamma}m_{\gamma,U}^{-s}
\end{align*}
From this, the assertion follows.
\end{proof}

\begin{lem}\label{InvariantTVBounded}
Let $A$ be symmetric with non-negative entries and $A_{UU}=0$ for all $U\in\mathcal{U}$, and let $\alpha\in\mathds{R}$, $s>1$.
Then the operator $\mathcal{D}_{A,\Gamma}^\alpha$ is a 
self-adjoint bounded linear operator on $L_s^2(\Omega)^\Gamma$.
\end{lem}

\begin{proof}
A simple calculation shows that
\[
\langle f,\mathcal{A}_\Gamma^\alpha g\rangle_s
=\langle \mathcal{A}_\Gamma^\alpha,g\rangle_s
\]
As in the case of Lemma \ref{self-adjoint}, this implies self-adjointness.

\smallskip
Boundedness now follows from Lemma \ref{Gammabounded}
in the same manner as the proof of Corollary \ref{bounded}.
\end{proof}

We now define the $\Gamma$-invariant wavelets as follows:
\[
\Psi_{j,U}^\Gamma(x)
=\sum\limits_{\gamma\in\Gamma}
\mu(\gamma U)^{-\frac12}\psi_{j,\gamma U}(\gamma^{-1}x)
\]
and their normalisations:
\[
\psi_{j,U}^\Gamma(x)
=\norm{\Psi_{j,U}^\Gamma}_s^{-1}\Psi_{j,U}^\Gamma
\]
where $j\in\tau(k)$, $U\in\mathcal{V}\in\mathcal{C}$, with
$\mathcal{C}$ a strong maximal chain in $\Cov(\mathscr{F})$ containing $\mathcal{U}$.

\begin{thm}\label{invariantEigenBasis}
Let $\alpha\in\mathds{R}$, and $s>1$. 
Let $\mathcal{U}\in\Cov(\mathscr{F})$  be verticial,
and let $\mathcal{U}$-matrix $A=(A_{UV})$ be symmetric matrix 
with non-negative entries such that $A_{UU}=0$ for all $U\in\mathcal{U}$.
Let $\mathcal{C}$ be a strong chain of admissible coverings of $\mathscr{F}$ such that $\mathcal{U}\in\mathcal{C}$ is at the top, and that $\mathcal{C}$ is maximal with this property.
Then the $\Gamma$-invariant wavelets $\psi_{j,U}^\Gamma$ with
$j\in\tau(k)$, $U\in\mathcal{V}\in\mathcal{C}$, and $\gamma\in\Gamma$
together with the 
functions
\[
e_\Gamma(x)=\sum\limits_{U\in\mathcal{U}}
\sum\limits_{\gamma\in\Gamma}e_U\Omega(x\in\gamma U)
\]
where 
 $e=(e_U)$ is an  eigenvector of the Laplacian of the
matrix
\[
B^\alpha=(B_{UV}),\quad B_{UV}
=\begin{cases}
0,&U=V\\
A_{UV}\absolute{x-y}_K^\alpha,&\text{otherwise}
\end{cases}
\]
with $x\in U$, $y\in V$, such that $\norm{e_\Gamma}_s=1$, 
are an
orthonormal basis of 
$L^2(\Omega)^\Gamma$ consisting of eigenfunctions of $\mathcal{D}_{A,\Gamma}^\alpha$.
\end{thm}

\begin{proof}
First the assertion about the wavelets.
We compute, assuming that $W\in\mathcal{V}\in\mathcal{C}$,  $W\subseteq U\in\mathcal{U}$:
\begin{align*}
\int_{\Omega,s}&A_{A,\Gamma}^\alpha(x,y)\,dy\cdot\Psi_{j,W}^\Gamma(x)
\\
&=\sum\limits_{V\in\mathcal{U}}
A_{UV}
\sum\limits_{\gamma,\delta\in\Gamma}
\int_{\delta V,s}\absolute{\gamma^{-1}x-\delta^{-1}y}_K^\alpha\,dy
\cdot\Omega(x\in\gamma U)
\cdot\Psi_{j,W}^\Gamma(x)\\
&=\sum\limits_{V\in\mathcal{U}}A_{UV}
\sum\limits_{\gamma,\delta\in\Gamma}
m_{\delta,V}^{-s}\int_V\absolute{\gamma^{-1}x-z}_K^\alpha\absolute{\delta'z}_K\,dz
\cdot\Omega(x\in\gamma U)
\cdot\Psi_{j,W}^\Gamma(x)
\end{align*}
Now,
\begin{align*}
\left|m_{\delta,V}^{-s}
\int_V\absolute{\gamma^{-1}x-z}_K^\alpha\absolute{\delta'z}_K\,dz\right|&
\le
m_{\delta,V}^{-s}\,\mu(V)\max\limits_{z\in V}
\absolute{\gamma^{-1}x-z}^{\alpha}\max\limits_V\absolute{\delta'}
\\
&\le\begin{cases}
C_V(\alpha)\cdot\max\limits_V\absolute{\delta'},&\max\limits_V\absolute{\delta'}<1\\
C_V(\alpha)\cdot\max\limits_V\absolute{\delta'}^{1-s},&\text{otherwise}
\end{cases}
\end{align*}
with $C_V(\alpha)=\mu(V)\max\limits_{z\in V}\absolute{\gamma^{-1}x-z}^\alpha$.
By Lemma \ref{absderivative}, it follows that 
\[
\sum\limits_{\delta\in\Gamma}\int_{\delta V,s}\absolute{\gamma^{-1}x-\delta^{-1}y}_K^\alpha\,dy
\]
converges for $\alpha\in\mathds{R}$, because $s>1$.

\smallskip
Now, by Lemma \ref{absConst}, we have that
$\absolute{x-y}_K$ is constant on $U\times V$ for $U\neq V$.
Hence,
\[
\sum\limits_{\delta\in\Gamma}\int_{\delta V,s}
\absolute{\gamma^{-1}x-\delta^{-1}y}_K^\alpha\,dy\;
\Omega(x\in\gamma W)
\]
does not depend on $x\in\gamma W$. This integral does not depend on $\gamma\in\Gamma$, either.
Hence,
\begin{align*}
\int_{\Omega,s}&A^\alpha_{A,\Gamma}(x,y)\,dy\cdot\Psi_{j,W}^\Gamma(x)
\\
&=\sum\limits_{V\in\mathcal{U}}
A_{UV}\sum\limits_{\delta\in\Gamma}
m_{\delta,V}^s\int_{\delta V}\absolute{x-z}_K^\alpha\absolute{\delta'z}_K\,dz
\cdot\Psi_{j,W}^\Gamma(x)
\\
&=\deg_{A,\Gamma}^\alpha(U)\cdot \Psi_{j,W}^\Gamma(x)
\end{align*}
And a simple calculation yields that
\[
\int_{\Omega,s}A^\alpha_{A,\Gamma}(x,y)\Psi_{j,W}^\Gamma(y)\,dy
=0
\]
From this, it follows that $\Psi_{j,W}^\Gamma$ is an eigenfunction of $\mathcal{D}_{A,\Gamma}^\alpha$ with eigenvalue
$-\deg^\alpha_{A,\Gamma}(U)$.

\smallskip
Secondly, the assertion about eigenvectors of the Laplacian of $B^\alpha$ which is a well-defined matrix by Lemma \ref{absConst}.
Observe that $e_\Gamma(x)$ is 
the $\Gamma$-periodic extension of 
\[
e(x)=\sum\limits_{U\in\mathcal{U}}\mu(U)^{-1}e_U\Omega(x\in U),\;x\in\mathscr{F}
\]
and that the kernel  of $\mathcal{D}_{A,\Gamma}^\alpha$
is the doubly $\Gamma$-periodic extension of 
the $\pi$-adic kernel associated with the Laplacian of the
matrix $B^\alpha$ (via Lemma \ref{dict}).
This means, by the dictionary, that $e_\Gamma(x)$ is an eigenfunction of $\mathcal{D}_{A,\Gamma}^\alpha$, if and only if $e$ is an eigenvector of $B^\alpha$, as can be seen by using Corollary \ref{eigenCorrespondence}.

\smallskip
The proof of having an orthonormal basis consisting of  eigenfunctions from Theorem \ref{TVEigen} carries over to this case.
\end{proof}

\subsection{Cauchy problem on Mumford curves}

Here, we assume that $A$ is a $\mathcal{U}$-matrix which is symmetric, whose entries are non-negative such that $A_{UU}=0$, and $\alpha\in\mathds{R}$, $s>0$.

\smallskip
First, we make this important observation:

\begin{lem}\label{GammaBanach}
$C(\Omega)^\Gamma$
is a closed subspace of $C(\Omega)$ w.r.t.\ any norm for which the latter is a Banach space. 
\end{lem}

\begin{proof}
For this, let $f_n$ be a sequence in $C(\Omega)^\Gamma$ which converges to $f\in C(\Omega)$.
We now have
\[
f(\gamma x)=\lim\limits_{n\to \infty}f_n(\gamma x)
=\lim\limits_{n\to\infty}f_n(x)=f(x)
\]
Hence, $f$ is $\Gamma$-invariant, i.e.\ $f\in C(\Omega)^\Gamma$.
\end{proof}

\begin{lem}
Let $\epsilon>0$. Then $\epsilon\mathcal{D}_{A,\Gamma}^\alpha$ with $\alpha\in\mathds{R}$ generates a strongly continuous, positive contraction semigroup $\exp\left(t\epsilon\mathcal{D}_{A,\Gamma}^\alpha\right)$ on
$C(\Omega)^\Gamma$.
\end{lem}

\begin{proof}
Observe that by Lemma \ref{InvariantTVBounded}, $\mathcal{D}_{A,\Gamma}^\alpha$ is a bounded linear operator on $L^2(\Omega)^\Gamma$. Hence $\epsilon \mathcal{D}_{A,\Gamma}^\alpha$ is a closed linear operator.

\smallskip
We again go through the steps in the proof of \cite[Lemma 4.1]{Zuniga2020}.

\smallskip
The conditions 1.\ and 2.\ of Theorem \ref{Yosida-Hille} are straightforward. Condition 3.\ means that there exists $\eta>0$ such that for any $h\in C(\Omega)^\Gamma$ the equation
\begin{align}\label{Eigenequation}
\left(\eta-\epsilon\mathcal{D}_{A,\Gamma}^\alpha\right)u=h
\end{align}
has a solution $u\in C(\Omega)^\Gamma$. Let
\[
g(x)=\int_{\Omega,s}A^\alpha_{\Gamma}(x,y)\,dy
\]
In the proof of Theorem \ref{invariantEigenBasis}, we have seen that
$g\in C(\Omega)^\Gamma$ 
is convergent for $x\in\Omega$.
Hence, $g$ is bounded, because $\mathscr{F}$ is compact.
This means that
there exists some $C>0$ such that
\[
\norm{g}_\infty\le C
\]
Now, rewrite equation (\ref{Eigenequation}) in this way:
\[
u(x)-\epsilon\int_{\Omega,s}u(y)
\frac{A^\alpha_{\Gamma}(x,y)}{\eta+\epsilon g(x)}\,dy
=\frac{h(x)}{\eta+\epsilon g(x)}
\]
Observe that the right hand side is in $C(\Omega)^\Gamma$.
Now, the operator
\[
T\colon C(\Omega)^\Gamma\to C(\Omega)^\Gamma,\;
u(x)\mapsto \epsilon\int_{\Omega,s}u(y)\frac{A^\alpha_{\Gamma}(x,y)}{\eta+\epsilon g(x)}\,dy
\]
satisfies
\[
\norm{T}\le\frac{\epsilon C}{\eta}
\]
By taking $\eta>\epsilon C$, the operator $I-T$ has an inverse in
$C(\Omega)^\Gamma$ (use Lemma \ref{GammaBanach}). And notice that $\eta$ is independent of $h$.

\smallskip
It follows that the closed operator $\mathcal{D}_{A,\Gamma}^\alpha$ generates a semigroup $Q_t$ ($t\ge 0$) satisfying the conditions of Yosida-Hille (Theorem \ref{Yosida-Hille}). On the other hand,
since $\epsilon \mathcal{D}_{A,\Gamma}^\alpha$ is a linear bounded operator on a Banach space, it follows that 
$\exp\left( t\epsilon\mathcal{D}_{A,\Gamma}^\alpha\right)$ is a uniformly continuos semigroup, and by using that the infinitesimal generators of $\exp\left(t\epsilon\mathcal{L}_{A,\Gamma}^\alpha\right)$ and $Q_t$ ($t\ge0$) agree, we have
\[
Q_t=\exp\left(t\epsilon\mathcal{L}_{A,\Gamma}^\alpha\right)
\]
for $t\ge 0$, cf.\ e.g.\ \cite[Thm.\ 1.2 and Thm.\ 1.3]{Pazy1983}
\end{proof}

We are now going to look at the following Cauchy problem which makes sense due to the completeness property of $C(\Omega)^\Gamma$ (cf.\ Lemma \ref{GammaBanach}):

\begin{task}[Cauchy Problem]\label{CPMumf}
Find $h(t,x)\in C^1\left((0,\infty),C(\Omega)^\Gamma\right)$ such that
\begin{align}\label{heateqMumf}
\left(\frac{\partial}{\partial t}-\epsilon\mathcal{D}_{A,\Gamma}^\alpha\right) h(t,x)=0
\end{align}
for $t\ge0$, $x\in\Omega$ which satisfies the initial condition
\[
h(0,x)=h_0(x)
\]
where $h_0\in C(\Omega)^\Gamma$ is fixed.
\end{task}

As the corresponding semigroup is Feller, it describes a 
$\pi$-adic heat equation on the Mumford curve $X=\Omega/\Gamma$.
Consequently
there is a $\pi$-adic diﬀusion process in $X$ attached to the diﬀerential equation (\ref{heateqMumf}).

\begin{thm}
There exists a probability measure $p_t(x,\cdot)$ with $t\ge0$, $x\in\Omega$ on the Borel $\sigma$-algebra of $\Omega$ such that the Cauchy Problem (Task \ref{CPMumf}) has a unique solution of the form
\[
h(t,x)=\int_{\Omega,s}h_0(y)p_t(x,dy)
\]
In addition, $p_t(x,\cdot)$ is the transition function of a Markov process whose paths are right continuous and have no discontinuities other than jumps.
\end{thm}

\begin{proof}
The proof of \cite[Thm.\ 4.2]{Zuniga2020} carries over word for word, relying also on Lemma \ref{GammaBanach}.
\end{proof}

In fact, the operator $\epsilon\mathcal{D}_{A,\Gamma}^\alpha$ generates a convolution semigroup, as we will see below. First, set
\begin{align*}
\Lambda_{A,\Gamma}^\alpha(x,y)
=\sum\limits_{\gamma,\delta\in\Gamma}
\Lambda_{A,\mathcal{U}}^\alpha(\gamma^{-1}x,\delta^{-1}y)
\end{align*}

Now, convolution in this setting means
\[
\Lambda_{A,\Gamma}^\alpha(x,y)
*_yf(y)
=\int_{\Omega,s}\Lambda_{A,\Gamma}^\alpha(x,x-y)f(y)\,dy
\]

\begin{lem}
It holds true that
\[
\mathcal{D}_{A,\Gamma}^\alpha f
=\Lambda_{A,\Gamma}^\alpha(\cdot,y)*_yf(y)
\]
for $f\in L^2_s(\Omega)^\Gamma$ and
$\alpha\in\mathds{R}$. In other words,
the operator $\mathcal{D}_{A,\Gamma}^\alpha$ acts by convolution on $L^2_s(\Omega)^\Gamma$.
\end{lem}

\begin{proof}
This is a simple calculation.
\end{proof}

\begin{cor}
The solution of Task \ref{CPMumf} 
is obtained by convolution with the function
\[
\exp\left(t\epsilon\Lambda_{A,\Gamma}^\alpha(x,y)\right)
\]
in the variable $y$.
\end{cor}

\begin{proof}
This is immediate.
\end{proof}

\subsection{Canonical heat operators on Mumford curves and spectral gap}

We assume now that $\mathcal{U}$ is a verticial covering of $X$
by rational affinoid domains,
and let $G_\mathcal{U}$ be the corresponding reduction graph.
Recall that this means that 
\[
\mathcal{U}(K)=\mathset{U(K)\colon U\in \mathcal{U}}
\]
is a disjoint covering of $X(K)$ by holed discs in $K$.
Further, assume that $A$ is a $\mathcal{U}(K)$-matrix which is at the same time an adjcacency matrix of $G_{\mathcal{U}}$.
\smallskip
The operator
\[
\mathcal{D}_{\mathcal{U}(K),\Gamma}^\alpha
=\mathcal{D}_{A,\Gamma}^\alpha
\]
on $L^2_s(\Omega)^\Gamma$ is called \emph{canonical}, and the heat equation
\begin{align}\label{canonicalHE}
\frac{\partial}{\partial t}u(t,x)=\epsilon\mathcal{D}_{\mathcal{U}(K)}^\alpha u(t,x)=0
\end{align}
is a \emph{canonical heat equation} on $X$.

\smallskip
By taking a lift of a maximal subtree of $G_{\mathcal{U}}$ into the Bruhat-Tits tree for $K$, we identify $\mathcal{U}$ with a verticial covering of a fundamental domain in $\Omega$ which is clopen compact.

\smallskip
We now look at  the following decomposition of $L^2(\Omega)_s$ 
into eigenspaces of the canonical Taibleson-Vladimirov operators on $X$ used in
(\ref{canonicalHE}). Let
\[
L^2_{\mathcal{U}(K)}(\Omega)^\Gamma
\]
denote the space spanned by the $\Gamma$-invariant $\mathcal{C}$-wavelets, where $\mathcal{C}$ is as in Theorem \ref{invariantEigenBasis}.

\begin{thm}\label{InvDecomposition}
It holds true that
\[
L^2_s(\Omega)^\Gamma
\cong L^2_{\mathcal{U}(K)}(\Omega)^\Gamma
\oplus A(G_{\mathcal{U}},\mathds{C})/c(G_{\mathcal{U}},\mathds{C})
\oplus\mathds{C}
\]
\end{thm}

\begin{proof}
This follows in the same manner as in the end of Section \ref{sec:wavelets}, by observing that
\[
L^2_s(\Omega)^\Gamma/ L^2_{\mathcal{U}(K)}(\Omega)^\Gamma
\cong H(G_{\mathcal{U}},\mathds{C})
\]
due to the surjective linear map
\[
\sigma\colon L^2_s(\Omega)^\Gamma\to H(G_{\mathcal{U}},\mathds{C}),
\;f\mapsto
\left(U\mapsto\sum\limits_{\gamma\in\Gamma}
\int_{\gamma U,s}f(x)\,dx\right)
\]
As $G_{\mathcal{U}}$ is connected, we have
\[
\ker(\phi\circ\sigma)=L_0^2(\Omega)^\Gamma
\]
where $\phi$ is from the exact sequence 
\begin{align}\label{ExactSeq}
\xymatrix{
0\ar[r]&c(G_{\mathcal{U}},\mathds{C})\ar[r]&
A(G_{\mathcal{U}},\mathds{C})\ar[r]^d
&H(G_{\mathcal{U}},\mathds{C})\ar[r]^{\quad\phi}
&\mathds{C}\ar[r]&0
}
\end{align}
and 
\[
L^2_0(\Omega)^\Gamma
=\mathset{f\in L^2_s(\Omega)^\Gamma\colon
\int_{\Omega,s}f(x)\,dx=0}
\]
This means that
\[
L^2_s(\Omega)^\Gamma/L_0^2(\Omega)^\Gamma
\cong\mathds{C}
\]
And the exact sequence also yields that
\[
H(G_{\mathcal{U}},\mathds{C})\cong
A(G_{\mathcal{U}},\mathds{C})/c(G_{\mathcal{U}},\mathds{C})
\oplus\mathds{C}
\]
where the first direct summand on the right hand side equals
\[
A(G_{\mathcal{U}},\mathds{C})/c(G_{\mathcal{U}},\mathds{C})\cong 
d(A(G_{\mathcal{U}},\mathds{C})\cong
L_0^2(\Omega)^\Gamma/L^2_{\mathcal{U}(K)}(\Omega)^\Gamma
\]
Also, we have
\[
\ker(\sigma)=L^2_{\mathcal{U}(K)}(\Omega)^\Gamma
\]
so that
\[
H(G_{\mathcal{U}},\mathds{C})
\cong L^2_s(\Omega)^\Gamma/L^2_{\mathcal{U}(K)}(\Omega)^\Gamma
\]
This proves the assertion.
\end{proof}

\begin{rem}
The exact sequence (\ref{ExactSeq}) appears in the context of Mumford curves in
\cite{vdP1992}, where Theta functions for discontinuous groups are studied.
\end{rem}

The decomposition of $L^2(\Omega)^\Gamma$ as in
Theorem \ref{InvDecomposition} together with
Theorem \ref{invariantEigenBasis} applied to an adjacency matrix $A$ of the reduction graph $G_{\mathcal{U}}$ of a Mumford curve $X$ shows that the non-zero spectrum of $\mathcal{D}_{A,\Gamma}^\alpha$ 
consists of the negative eigenvalues of the Laplacian associated with $A$
and the quantities
\[
-\deg_{A,\Gamma}^\alpha(U)
\]
for $U\in\mathcal{U}$. We will now prove that the spectral gap of $\mathcal{D}_{A,\Gamma}^\alpha$ can become arbitrarily small if $X$ varies over all Mumford curves having the same combinatorial reduction graph.

\begin{ass}\label{assumption}
In order to eliminate the dependence of the spectrum
on the volume of the fundamental domain $\mathscr{F}$, we assume that for a given Mumford curve $X$, it is a holed disc inside $O_K$ with maximal possible volume under such fundamental domains.
\end{ass}

\begin{lem}\label{degreeTVGamma}
Let $X=\Omega/\Gamma$ be a Mumford curve of genus $g\ge 1$
with fixed reduction graph $G$ which is assumed to be simple.
Let $\alpha\in\mathds{R}$, and $s>1$. Let $A=(A_{UV})$ be the combinatorial adjacency matrix of $G$. 
Consider the sequence 
\[
p^{\frac{\beta_{\gamma,V}}{e}}=\max\limits_V\absolute{\gamma'}_K,\quad\gamma\in\Gamma
\]
in its natural ordering, and let $\lambda_V\ge 1$ be the smallest gap within the sequence $\beta_{\gamma,V}\in\mathds{Z}$. Then, 
it holds true that
\[
\deg_{A,\Gamma}^\alpha(U)
\le\deg_{G}(U)\max\limits_{A_{UV}\neq0}
\mathset{\mu(V)d(U,V)^\alpha}C(s,\lambda_V)
\]
with
\[
C(s,\lambda_V)=
\left(\frac{1}{1-p^{-\frac{\lambda_V}{e}}}
+\frac{p^{-\frac{\lambda_V(s-1)}{e}}}{1-p^{-\frac{\lambda_V(s-1)}{e}}}\right)
\]
for $U\in\mathcal{U}$.
\end{lem}

\begin{proof}
We will give an upper bound for $\deg_{A,\Gamma}^\alpha(U)$
for $U\in\mathcal{U}$ as follows, where $d_K(U,V)$ denotes the $\pi$-adic distance between the sets $U$ and $V$:
\begin{align*}
\deg_{A,\Gamma}^\alpha(U)
&=\sum\limits_{V\in\mathcal{U}}A_{UV}\sum\limits_{\gamma\in\Gamma}
\int_{\gamma V,s}\absolute{x-\gamma^{-1}y}_K^\alpha\,dy
\\
&=\sum\limits_{V\in\mathcal{U}}
A_{UV}\sum\limits_{\gamma\in\Gamma}
m_{\gamma,V}^{-s}\int_V\absolute{x-y}_K^\alpha\absolute{\gamma'y}_K\,dy
\\
&\le \sum\limits_{V\in\mathcal{U}}A_{UV}\mu(V)d_K(U,V)^\alpha
\sum\limits_{\gamma\in\Gamma}
\frac{\max\limits_V\absolute{\gamma'}_K}{
\max\mathset{1,\max\limits_V\absolute{\gamma'}_K}^{s}}
\end{align*}
Now, in
\[
S_V(s)=\sum\limits_{\gamma\in\Gamma}\frac{\max\limits_V\absolute{\gamma'}_K}{\max\mathset{1,\max\limits_V\absolute{\gamma'}_K}^s}
\]
each numerator is a  power of $p^f$, and the denominator is a  power of $p^{sf}$, so each summand can be written like this:
\[
\begin{cases}
p^{-\frac{\nu_{V,n}}{e}},&m_{\gamma,V}=1\\
p^{-\frac{\nu_{V,n}(s-1)}{e}},&m_{\gamma,V}>1
\end{cases}
\]
and $\nu_{V,n}$ is an incresasing sequence of natural numbers
with $\nu_{V,0}=0$.
Hence,
\[
S_V(s)=\sum\limits_{n\ge 0}p^{-\frac{\nu_{V,n}}{e}}
+\sum\limits_{n>0}p^{-\frac{\nu_{V,n}(s-1)}{e}}
\]
If $\lambda_V$ is the minimal gap in the sequence 
$\nu_{V,n}=\beta_{\gamma,V}$ for suitable $\gamma\in\Gamma$, then we obtain
\begin{align*}
S_V(s)&\le
\sum\limits_{n\ge 0}p^{-\frac{n\lambda_V}{e}}
+\sum\limits_{n>0}p^{-\frac{n\lambda_V(s-1)}{e}}
=C(s,\lambda_V)
\end{align*}
Hence, we obtain
\begin{align*}
\deg_{A,\Gamma}^\alpha(U)
\le\deg_{G}(U)\max\limits_{A_{UV}\neq0}
\mathset{\mu(V)d(U,V)^\alpha}
C(s,\lambda_V)
\end{align*}
as asserted.
\end{proof}

A graph is called \emph{stable} if every vertex not attached to a loop-edge
is attached to at least three edges.
By the stable reduction Theorem \cite{DM1969}, among the possible intersection graphs of the special fibre of a given projective algebraic curve, the stable intersection graph is unique.

\begin{cor}
Let $\alpha\in\mathds{R}$, $s>1$, and let  $M_g(G)$ be the set of isomorphism classes of Mumford curves of genus $g\ge1$ defined over $K$ having fixed stable reduction graph $G$. Then, under Assumption \ref{assumption},  for every $\epsilon >0$, $M_G$ contains 
a curve $X=\Omega/\Gamma$ such that the spectral gap of the
operator
$\mathcal{D}_{A(G),\Gamma}^\alpha$ is smaller than $\epsilon$, where $A(G)$ is the combinatorial (symmetric) adjacency matrix of $G$.
Furthermore, if $X$ is any curve from $M_g$, then under these assumptions, there exists a finite field extension $L$ such that
$X(L)$  satisfies
\[
\lim_{p\to \infty}\deg_{A,\Gamma}^\alpha(U)>0
\]
for some $U\in\mathcal{U}$ and a verticial covering $\mathcal{U}$ of $X(L)$.
\end{cor}

\begin{proof}
By Lemma \ref{degreeTVGamma}, there is an upper bound of 
$\deg_{A,\Gamma}^\alpha$ which is a multiple of
\[
\Lambda_X(U)=\max\limits_{A_{UV}\neq0}\mathset{\mu(V)d(U,V)^\alpha}
\]
For each Mumford curve in $M_g(G)$, choose the fundamental domain
$\mathscr{F}_X$ such that the reduction trees of $\mathscr{F}_X$ are isomorphic. For the vertical covering $\mathcal{U}_X$
of $\mathscr{F}_X$, let 
$U\in\mathcal{U}_X$ such that $\Lambda_X(U)$ is minimal. The reduction of the fundamental domain $\mathscr{F}_X$ being a finite subtree $T_X$ of the Bruhat-Tits tree, we may now vary $X$ within $M_g(G)$,
such that the edges of $T_X$ become longer beyond any bound. Under Assumption \ref{assumption}, this implies that there exists $U\in\mathcal{U}$ having a neighbour $V\in\mathcal{U}$ such that
$\mu(V)$ becomes  smaller than any given bound. This implies that
$\Lambda_X(U)$ and also $\deg_{A,\Gamma}^\alpha(U)$ becomes smaller than any given bound. This proves the assertion about the spectral gap, as $-\deg_{A,\Gamma}^\alpha(U)$ is an eigenvalue of $\mathcal{D}_{A,\Gamma}^\alpha$ by Theorem \ref{invariantEigenBasis}.

\smallskip
As to the second assertion, we will look at 
the term in $\deg_{A,\Gamma}^\alpha(U)$ corresponding to $\gamma=\id$. We need $U\in\mathcal{U}$ such that there exists an adjacent $V\in\Star(U)$ such that
$d_p(U,V)=1$ and $\mu(V)=1-\ell p^{-\frac{1}{e_{L}}}$ for some $\ell\in\mathset{1,\dots,p^{f_L}-1}$, where $e_L$ and $f_L$ are the ramification index and the residue field degree of $L/\mathds{Q}_p$, respectively. Under Assumption \ref{assumption} and after a possible finite ramified extension $L$ of $K$, such $U\in\mathcal{U}$ exist.
The isomorphism class of $X$ contains a curve such that $V$ is obtained by intersecting the sphere $\absolute{x}_L=1$ with other spheres of radius $1$. This implies that
\[
\int_V\absolute{x-y}_L^\alpha\absolute{y}_L\,dy=\mu(V)=1-\ell p^{-f}
\]
and as $m_{\id,V}=1$, we have now proven that
\[
\deg_{A,q}^\alpha(U)=C_U+O\left(p^{-nf}\right)
\]
for some natural $n\ge 1$, and where $C_U>1$.
For this $U$, we clearly have
\[
\lim_{p\to\infty}\deg_{A,\Gamma}^\alpha(U)=C_U>0
\]
as asserted.
\end{proof}

\subsection{Degree eigenvalues of canonical heat operators on Tate curves}

Here, we will exhibit an explicit calculation of 
\[
\deg_{A,q}^\alpha(U)=\deg_{A,\Gamma}^\alpha(U)
\]
under Assumption \ref{assumption} for Tate curves.
Let
$X=K^\times/\langle q\rangle$ with $\absolute{q}=p^{-\frac{n+1}{e}}$
be a Tate curve. Let $\mathcal{U}$ be the covering of the fundamental domain, which is an annulus, as follows:
\[
\mathcal{U}=\mathset{U_0,\dots,U_n}
\]
and
\[
U_i=\mathset{\absolute{x}_K=p^{-\frac{i}{e}}},\quad
i\equiv 0,\dots,n+1
\]
with $n\ge 2$. The latter condition guarantees that the reduction graph is a simple graph.

\begin{thm}
It holds true that
\[
\deg_{A,q}^\alpha(U_i)
=C_\alpha(i)\cdot S_q(s)
\]
with
\[
C_\alpha(i)=\begin{cases}
\left(p^{-i\alpha}p^{-(i+1)f)}+p^{-(i-1)\alpha}p^{-(i-2)f}\right)\left(1-p^{-f}\right),&i=1,\dots,n-1\\
\left(1+p^{-(n-1)\alpha}p^{-(n-2)f}\right)\left(1-p^{-f}\right),&i=n
\\
\left(1+p^{-(n-1)f}\right)\left(1-p^{-f}\right),&i=0
\end{cases}
\]
and
\[
S_q(s)=\frac{1}{1-\absolute{q}}+\frac{\absolute{q}^{s-1}}{1-\absolute{q}^{s-1}}
\]
where $\alpha\in\mathds{R}$ and $s>1$.
\end{thm}

By Theorem \ref{invariantEigenBasis}, this gives an explicit formula for the eigenvalues corresponding to  the  wavelet eigenfunctions of $\mathcal{D}_{A,\Gamma}^\alpha$, which are the negatives of those degrees.

\begin{proof}
We have, where natural indices are read modulo $n+1$,
\begin{align*}
\deg_{A,q}^\alpha(U_i)
&=\sum\limits_{n\in\mathds{Z}}
\int_{q^n U_{i+1},s}\absolute{x-q^{-n}y}_K^\alpha\,dy
+\int_{q^nU_{i-1},s}\absolute{x-q^{-n}y}_K^\alpha\,dy
\\
&=\left(\mu(U_{i+1})d(U_i,U_{i+1})^\alpha
+\mu(U_{i-1})d(U_i,U_{i-1})^\alpha\right)
\sum\limits_{n\in\mathds{Z}}m_{q^n}^{-s}\absolute{q}^n_K
\end{align*}
with
\[
m_{q^n}=\max\mathset{1,\absolute{q}^n}
\]
The factor before the sum equals
\[
C_\alpha(i)
\]
and the sum equals
\[
\sum\limits_{n\ge 0}\absolute{q}^n+\sum\limits_{n>0}\absolute{q}^{n(s-1)}=S_q(s)
\]
as asserted.
\end{proof}

\begin{exa}
We will look at a Tate curve over $\mathds{Q}_p$.
Let $n=2$ in the setting of this subsection.
The matrix $B^\alpha$ for $\alpha=1$ is
\[
B^1=\begin{pmatrix}
0&1&1\\
1&0&p^{-1}\\
1&p^{-1}&0
\end{pmatrix}
\]
Its non-negative Laplacian eigenvalues are
\[
\lambda_1=-3,\quad \lambda_2=-(1+2p^{-1})=-\frac{p+2}{p}
\]
If $s>>1$, we have
\begin{align*}
  \deg_{A,q}^1(U_0)&\approx \left(p^{-1}+p^{-2}\right)\cdot\frac{1-p^{-1}}{1-p^{-3}}
  =\frac{p^2-1}{p^3-1}
\\
\deg_{A,q}^1(U_1)&\approx
\left(p^{-3}+1\right)\cdot\frac{1-p^{-1}}{1-p^{-3}}
=\frac{p^4-p^3+p-1}{p^4-p}
\\
\deg_{A,q}^1(U_2)&\approx
\left(1+p^{-2}\right)\cdot\frac{1-p^{-1}}{1-p^{-3}}
=\frac{p^3-p^2+p-1}{p^3-1}
\end{align*}
Observe that $\lambda_2$ tends to one for $p\to\infty$.
The same holds true for $\deg_{A,q}^1(U_1)$ and $\deg_{A,q}^1(U_2)$.
However, $\deg_{A,q}^1(U_0)$ tends to zero for $p\to\infty$.
Furthermore, all degree values for $s>>0$  are smaller than one, whereas $-\lambda_2$ is larger than one.

\smallskip
The eigenvalues for $p=2,3,5$ are shown in Table \ref{table}.

\begin{table}[h]
\[
\begin{array}{|c|c|c|c|c|c|}\hline
&U_0&U_1&U_2&-\lambda_1&-\lambda_2\\\hline
\rule{0pt}{5mm}p=2&\frac37\approx0.429&\frac{9}{14}\approx0.643&\frac57\approx0.714&3&2
\\
\rule{0pt}{5mm}p=3&\frac{4}{13}\approx0.308
&\frac{28}{39}\approx0.718&\frac{10}{13}\approx0.792
&3&\frac53\approx1.667
\\
\rule[-3mm]{0pt}{8mm}p=5&\frac{6}{31}\approx0.194&\frac{126}{155}\approx0.813&\frac{26}{31}\approx0.839&3&\frac75=1.4
\\\hline
\end{array}
\]
\caption{The Taibleson-Vladimirov degrees for the first three primes, and $s>>1$, compared with  Laplacian eigenvalues $\lambda_1,\lambda_2$.}
\label{table}
\end{table}
\end{exa}

We conclude with the following two questions:

\begin{quest}
Is it always true that, in the case of a Mumford curve, the Laplacian spectral gap of the adjacency matrix $B^\alpha$ tends to one for $p\to\infty$?
\end{quest}

In our example, the spectral gap of $\mathcal{D}_{A,\Gamma}^\alpha$ equals the absolute value of a wavelet eigenvalue.

\begin{quest}
Does there exist a Mumford curve of genus $g\ge 1$ such that, under Assumption \ref{assumption} and with $A$ a combinatorial adjacency matrix, the spectral gap of $\mathcal{D}_{A,\Gamma}^\alpha$ equals the spectral gap of $B^\alpha$?
\end{quest}

\section*{Acknowledgements}

We are indebted to Wilson Z\'u\~{n}iga-Galindo for presenting the article \cite{Zuniga2020} at the p-adics.2019 conference in 
Covilh\~{a}, Portugal, and for further fruitful discussions towards attempts in constructing integral operators on $p$-adic manifolds, as well as for straightening out the nomenclature. Also David Weisbart is to be credited for inciting discussions about topics in $p$-adic analysis, where my attention was brought towards the major helpful idea towards constructing Z\'{u}\~{n}iga operators on Mumford curves, and also for suggestions for this article.

\bibliographystyle{plain}
\bibliography{biblio}

\begin{thebibliography}{10}

\bibitem{AK2010}
S.~Albeverio and S.V. Kozyrev.
\newblock Pseudodifferential $p$-adic vector fields and pseudodifferentiation
  of a composite $p$-adic function.
\newblock {\em $p$-Adic Numbers, Ultrametric Analysis and Applications},
  2(1):21--34, 2010.

\bibitem{GenDiffMg}
P.E. Bradley.
\newblock Generalised diffusion on moduli spaces of $p$-adic {Mumford} curves.
\newblock {\em $p$-Adic Numbers, Ultrametric Analysis and Applications},
  12(2):73--89, 2020.

\bibitem{pWaveGraph}
P.E. Bradley.
\newblock $p$-{Adic} wave equations on finite graphs and {$T_0$}-spaces.
\newblock In W.~Zuniga and B.~Toni, editors, {\em Advances in Non-Archimedean
  Analysis and Applications}, STEAM-H, pages 275--295. Springer, New York,
  2021.

\bibitem{CZ2018}
L.F. Chacón-Cortés and W.A. Zúñiga-Galindo.
\newblock Heat traces and spectral zeta functions for $p$-adic {Laplacians}.
\newblock {\em St. Petersburg Math. J.}, 29:529--544, 2018.

\bibitem{DM1969}
P.~Deligne and D.~Mumford.
\newblock The irreducibility of the space of curves of a given genus.
\newblock {\em Publ. Math. I.H.E.S.}, 36:79--109, 1969.

\bibitem{EK1986}
S.N. Ethier and T.G. Kurtz.
\newblock {\em Markov Processes - Characterization and Convergence}.
\newblock Wiley Series in Probability and Mathematical Statistics. John Wiley
  \& Sons, New York, 1986.

\bibitem{FP2004}
J.~Fresnel and M.~{van der Put}.
\newblock {\em Rigid Analytic Geometry and its Applications}, volume 218 of
  {\em Progress in Mathematics}.
\newblock Birkh\"auser, Boston, Mass., 2004.

\bibitem{GvP1980}
L.~Gerritzen and M.~{van der Put}.
\newblock {\em Schottky Groups and Mumford Curves}.
\newblock Lecture Notes in Mathematics, vol. 817. Springer, Heidelberg, New
  York, 1980.

\bibitem{XK2005}
A.Yu. Khrennikov and S.V. Kozyrev.
\newblock Pseudodifferential operators on ultrametric spaces and ultrametric
  wavelets.
\newblock {\em Izvestiya: Mathematics}, 69(5):989--1003, 2005.

\bibitem{Kozyrev2002}
S.V. Kozyrev.
\newblock Wavelet theory as $p$-adic spectral analysis.
\newblock {\em Izv. Math.}, 66(2):367--376, 2002.

\bibitem{Mumford1972}
D.~Mumford.
\newblock An analytic construction of degenerating curves over complete local
  rings.
\newblock {\em Compositio Mathematica}, 24(2):129--174, 1972.

\bibitem{Pazy1983}
A.~Pazy.
\newblock {\em Semigroups of Linear Operators and Applications to Partial
  Diﬀerential Equations}.
\newblock Applied Mathematical Sciences, vol. 44. Springer-Verlag, New York,
  1983.

\bibitem{Roquette1970}
P.~Roquette.
\newblock {\em Analytic theory of elliptic functions over local fields}.
\newblock Hamburger Math. Einzelschriften, Neue Folge, Heft 1. Vandenhoeck \&
  Ruprecht, G\"ottingen, 1970.

\bibitem{Taibleson1975}
M.H. Taibleson.
\newblock {\em Fourier analysis on local fields}.
\newblock Princeton University Press, Princeton, N.J., University of Tokyo
  Press, Tokyo, 1975.

\bibitem{TZ2018}
A.~Torresblanca-Badillo and W.A.~Z\'{u}\ {n}iga Galindo.
\newblock Ultrametric diffusion, exponential landscapes, and the first passage
  time problem.
\newblock {\em Acta Appl Math}, 157(1):93--116, 2018.

\bibitem{vdP1992}
M.~{van der Put}.
\newblock Discrete groups, {Mumford} curves and theta functions.
\newblock {\em Annales de la faculté des sciences de Toulouse 6 e série},
  1(3):399--438, 1992.

\bibitem{VVZ1994}
Vladimirov V.S., Volovich I.V., and Zelenov E.I.
\newblock {\em $p$-adic Analysis and mathematical physics}.
\newblock Series on Soviet and East European Mathematics, 1. World Scientific
  Publishing Co., Inc., River Edge, NJ, 1994.

\bibitem{Zuniga2015}
W.A. Zúñiga-Galindo.
\newblock The non-{Archimedean} stochastic heat equation driven by {Gaussian}
  noise.
\newblock {\em J. Fourier Anal. Appl.}, 21:600--627, 2015.

\bibitem{Zuniga2020}
W.A. Zúñiga-Galindo.
\newblock Reaction-diffusion equations on complex networks and {Turing}
  patterns via $p$-adic analysis.
\newblock {\em Journal of Mathematical Analysis and Applications},
  491(1):124239, 2020.

\end{thebibliography}

\end{document}